\documentclass[twoside,11pt]{article}

\usepackage{amsmath,authblk}
\usepackage{amsthm}
\usepackage{lastpage}
\usepackage{amssymb}
\usepackage{natbib, hyperref}
\usepackage{graphicx}
\usepackage{algorithm}% http://ctan.org/pkg/algorithms
\usepackage{algpseudocode}% http://ctan.org/pkg/algorithmicx
\usepackage{cleveref}
\usepackage{fullpage}
\algnewcommand\algorithmicinput{\textbf{INPUT:}}
\algnewcommand\INPUT{\item[\algorithmicinput]}
\algnewcommand\algorithmicoutput{\textbf{OUTPUT:}}
\algnewcommand\OUTPUT{\item[\algorithmicoutput]}

\usepackage{hyperref}[]
\hypersetup{
    colorlinks=true,
    linkcolor=blue,
    filecolor=magenta,      
    urlcolor=black,
      citecolor=blue,
    }

\newcommand{\ou}{\"{o}}
\newtheorem{definition}{Definition}
\newtheorem{lemma}{Lemma}
\newtheorem{proposition}{Proposition}
\newtheorem{corollary}{Corollary}
\newtheorem{remark}{Remark}
\newtheorem{theorem}{Theorem}

\newcommand{\dk}{\left(\frac{\delta}{\mathcal K}\right)^{1/\alpha}}
\newcommand{\dddelta}{\left(\frac{\delta}{\mathcal K}\right)^{1/\alpha}}
\newcommand{\C}{\mathcal C}
\newcommand{\K}{\mathcal K}
\newcommand{\hatsplit}{\widehat {\lambda^*}}

\begin{document}

\title{DBSCAN: Optimal Rates For Density-Based Cluster Estimation}
\author[1] {Daren Wang}
\author[2]{Xinyang Lu}
\author[3]{Alessandro Rinaldo}
 \affil[1]{\small Department of Statistics,  University of Chicago}
 \affil[2]{\small Mathematical Sciences Department,
       Lakehead University}
\affil[3]{\small Department of Statistics and Data Science,  Carnegie Mellon University}

\maketitle
\begin{abstract}
	
	We study the problem of optimal estimation of the density cluster tree under various smoothness assumptions on the underlying density.
	Inspired by the seminal work of 
\cite{chaudhuri2014consistent},
we formulate a new notion of clustering consistency which is better suited
to smooth  densities, and derive minimax rates for cluster tree
estimation under H\"{o}lder smooth densities of arbitrary degree. 	
	We present a computationally efficient, rate optimal cluster tree estimator based on simple extensions of the popular DBSCAN algorithm of \cite{ester1996density}.
	Our procedure relies on kernel density estimators and returns a sequence of nested random geometric graphs whose connected components form a hierarchy of clusters.
	The resulting optimal  rates for cluster tree estimation 
	depend on the degree  of smoothness of the underlying density and, interestingly,  match the minimax
	rates for density estimation under the sup-norm loss.  Our results complement and extend the analysis of the DBSCAN algorithm in \cite{sriperumbudur2012consistency}.
	Finally, we consider level set estimation and cluster consistency for densities with jump
	discontinuities. We demonstrate that the  DBSCAN algorithm  attains the minimax rate in terms of the jump size and sample size in this setting as well.
		\end{abstract}
	\

	\section{Introduction}
	\label{section:introduction}
	%!TEX root = ./dbscan_JMLR.tex

Clustering is one of the most basic and fundamental tasks in statistics and
machine learning, used
ubiquitously and extensively in the exploration and analysis of data.  The
literature on this topic is vast, and practitioners have at their disposal a
multitude of algorithms and heuristics to perform clustering on data of
virtually all
types. However, despite its importance and popularity, rigorous statistical
theories for clustering, leading to inferential procedures with provable
theoretical guarantees, have been
traditionally lacking in the literature. As a result, the practice of clustering, a central
tasks in the analysis and manipulation of data, still relies in many cases on methods and
heuristics of unknown or even dubious scientific validity. One of the most
striking instances of such a disconnect is the DBSCAN  algorithm  of \cite{ester1996density}, an
extremely popular
and relatively efficient  \citep[see][]{DBSCAN:revisited,wang2015design}  clustering methodology whose statistical properties
have been properly analyzed only very recently: see
\cite{sriperumbudur2012consistency}, \cite{jiang2017density} and \cite{ingo.bharath.new}.

In this paper, we provide a complementary and thorough study of DBSCAN, and show that this simple algorithm can deliver optimal statistical performance in 
{\it density-based} clustering. Density-based clustering \citep[see, e.g.,][]{hartigan1981consistency} provides a
general and rigorous 
probabilistic framework in which the clustering task is well-defined and
amenable to statistical analysis.  Given a  probability distribution $P$ on
$\mathbb{R}^d$ with  a corresponding continuous  density $p$ and a fixed threshold $\lambda \geq 0$,
the {\it $\lambda$-clusters} of $p$ are the connected components of  the upper
	$\lambda$-level set of $p$, the set  $\{ x \in \mathbb{R}^d \colon p(x) \geq \lambda \}$
of all points whose density values exceed the level $\lambda$. With this
definition, clusters are the high-density
regions, subsets of the support of $P$ with the largest probability content among
all sets of the same
volume.

As noted in \cite{hartigan1981consistency}, the hierarchy of
inclusions of all
clusters of $p$	is a tree structure indexed by
$\lambda > 0$, called the {\it
	cluster tree} of $p$. The chief goal of density clustering is to estimate the cluster tree of $p$, given an
i.i.d. sequence $\{X_i\}_{i=1}^n$ of points with common distribution $P$. 
A cluster tree estimator is also a tree structure, consisting
of a hierarchy of nested
subsets of the sample points, and typically relies on non-parametric
estimators of $p$ in order to determine which sample points belong to
high-density regions of $p$. A cluster tree estimator is deemed accurate
if, with high probability,  the hierarchy of clusters it encodes is close to the hierarchy that would have been  obtained should $p$ be known.

Density-based clustering, an instance of hierarchical clustering, enjoys several advantages: (1)
it imposes virtually no restrictions on the shape, size  and number of
clusters, at any level of the tree; 
(2) unlike {\it flat} (i.e.
non-hierarchical) clustering, it does not require a pre-specified number of
clusters as an input and in fact the number of clusters itself is a  quantity that may change depending on the level
of the tree; (3) it provides a multi-resolution representation of all the clustering
features of $p$ across all levels $\lambda$ at the same time; (4) it allows for an efficient representation and storage of the entire
tree of clusters with a compact data structure that can be easily accessed
and queried, and (5) the main object of interest for inference, namely  the cluster tree of $p$,
is a well-defined quantity.

Despite  the appealing properties of the density-based clustering framework,
a rigorous quantification of the statistical performance of this type of 
algorithms has proved difficult. Previous results by 
\cite{hartigan1981consistency} and then  \cite{penrose:95} have
demonstrated a weaker notion of consistency achieved by the popular
single-linkage algorithm. More recently \cite{chaudhuri2014consistent} have
developed a general framework for defining consistency of cluster tree
estimators based
on a separation criterion among clusters. The authors  further 
demonstrated that two graph-based algorithms, both based on $k$-nearest
neighbors graphs over the sample points,  achieve such consistency and provided minimax optimal consistency rates with respect to the  parameters specifying the amount of cluster separation. Such results hold with virtually no assumptions on the underlying
density. However, because of this generality, these consistency rates do not directly reflect any degree
of regularity or smoothness of the underlying density. In particular, it remains unclear whether cluster tree estimation would be easier with smoother densities.

In this paper we provide further contributions to the theory of density based
clustering by deriving novel, nearly minimax-optimal rates for cluster tree estimation that depend
explicitly on the smoothness of the underlying density function. Our results further confirm  that the smoother the density the faster the rate of consistency for the cluster tree estimation problem, a finding that is consistent with analogous results about non-parametric density estimation. Interestingly, our rates match those for estimating smooth densities in the
$L_\infty$ norm. To the best of our knowledge, this finding and the implication that density based clustering is no easier -- at least in our setting -- than density estimation, has not been rigorously shown before. In order to account explicitly for the smoothness of the density,  we have developed a new criterion for
cluster consistency
that is  better suited for smooth densities. In terms of procedures, we consider cluster tree estimators that arise from applying a very simple generalization of the well-known DBSCAN algorithm  and are computationally efficient. Furthermore, our DBSCAN-based estimator is minimax optimal over arbitrary smooth
densities according to our notion of consistency under appropriate conditions. 

%Those results have
%been generalized in \cite{balakrishnan2012}, where it is shown that 
%the main algorithm of \cite{chaudhuri2014consistent},  and a class of kernel density
%estimators, ensure similar
%consistency guarantees for cluster trees arising from probability 
%distributions supported over well-behaved manifolds, with consistency rates 
%depending on the reach of the manifold and its intrinsic dimension. 
%In both contributions, rates for cluster
%consistency are established 

\subsection*{Related work}

The idea of using the probability density function in order to study clustering structure dates back to 
  \cite{hartigan1981consistency}, who formalized 
the notion of clusters as the 
connected components of high density regions and of cluster tree. Much of the subsequent theoretical work focused on consistency for ``flat'' 
clustering at a fixed level, which effectively reduces to level set estimation. The literature on this topic is vast and offers  a multitude of results covering different settings and metric for consistency. See, e.g., \cite{penrose:95} \cite{polonik1995measuring}, \cite{tsybakov1997nonparametric}, \cite{cuevas1997plug}, \cite{ba2000set}, \cite{klemela2004complexity}, \cite{willett2007minimax}, \cite{hausdorff}, \cite{rigollet2009optimal}, \cite{rinaldo2010generalized}.
In contrast, there have been fewer contributions to the theory of practice of cluster tree estimation: see, e.g., \cite{ stuetzle2003estimating, stuetzle2010generalized}, \cite{klemela:09} and \cite{rinaldo2012stability}. The work of \cite{chaudhuri2014consistent} (see also \cite{kpotufe2011pruning}) represented a significant advance in the theory of density-based clustering, as it derived a new framework and consistency rates for cluster tree estimation. \cite{balakrishnan2012} generalized these results  to the  probability 
distributions supported over well-behaved manifolds, with consistency rates 
depending on the reach of the manifold and its intrinsic dimension. Corresponding  guarantees in Hausdorff distance have been recently obtained by \cite{jiang2017density}.
\cite{eldridge2015beyond} developed a unified theory for consistency in cluster tree estimation that encompasses the original framework of Hartigan while \cite{inference.tree} investigated the challenging problems of defining adequate metrics over the space of cluster tree and  of constructing confidence sets for cluster tree structures. \cite{chen2016density} provides bootstrap-based methods for constructing confidence sets for density level sets and for visualization of high-density clusters.   Recently, 
 \cite{jang2018dbscan++} proposed  a variant of the   DBSCAN algorithm with  both minimax clustering rate and sub-quadratic computational complexity while \cite{jiang2019robustness} studied DBSCAN  under possibly adversarial contamination of the input data.

In a parallel and important line of work, \cite{pmlr-v19-steinwart11a,steinwart:15} developed  a rigorous, measure-theoretic approach to density-based clustering whereby the cluster tree is recovered by estimating the lowest split level of the density and then proceeding recursively. The corresponding results demonstrate a direct link between density based clustering and optimal level set estimation. 
This approach was applied in \cite{sriperumbudur2012consistency} to show that the DBSCAN algorithm yield consistent estimator of density trees, a result that was then extended in \cite{ingo.bharath.new}  to allow for  more general, KDE-based procedures.
Our work built directly upon the contributions of \cite{chaudhuri2014consistent} and \cite{steinwart:15}.

\subsubsection*{Organization of the paper}

The rest of the paper is organized as follows. In \Cref{section:background}, we describe the DBSCAN algorithm and establish its connections with non-parametric density estimation.
In \Cref{section:holder} we introduce  a   new 
	notion of cluster consistency, called $\delta$-consistency that is tailored to H\"{o}lder-continuous densities.  We describe a  DBSCAN-based algorithm for clustered tree estimation that is computational efficient and delivers nearly optimal minimax rates that depend explicitly on the degree of smoothness of the underlying density, whereby cluster tree of smoother densities can be estimated at faster rates. Interestingly and, perhaps surprisingly, for the class of DBSCAN-based algorithms we consider, we observe a trade-off between statistical optimality and computational cost for smoother H\"{o}lder densities of degree $\alpha>1$. In these situations, minimax rates can still be achieved by our computationally efficient algorithm provided that the underlying density satisfies additional geometric regularity conditions around the split levels. Such conditions are relatively mild and have been exploited before; see in particular \cite{steinwart:15}.
Finally, in \Cref{section:gap} we consider a different scenario in which the
	underlying density exhibits jump discontinuities.  We are particularly interested in level set and cluster estimation at the jump, with the assumption that the size of the discontinuity is vanishing when $n \rightarrow \infty$ so that clustering becomes increasingly difficult. We show that, with suitable inputs, the DBSCAN algorithm returns a Devroye-Wise type of estimator which is minimax optimal for cluster recovery and level set estimation. In addition, we derive the minimax scaling for the size of the jump discontinuity.

%\textcolor{red}{ I think this paragraph is repeating the end of the previous paragraph.}
%We show that with suitable inputs, the DBSCAN algorithm returns a Devroye-Wise type of estimator which is minimax optimal for cluster recovery and level set estimation at the gap. The minimax optimality are derived in terms of the gap size and sampled size.

\subsubsection*{Notation}
We denote with $p $  a  density for the distribution $P$ of the i.i.d. sample $\{X_i\}_{i=1}^n \subset \mathbb R^d$.
For a constant $\lambda > 0$, we set $L(\lambda) = \{p\ge \lambda\}$ to be the $\lambda$-upper upper level set of the density $p$. 
We use $T_p$ to denote the cluster tree generated by the density $p$  and 
$\widehat T$ to density any estimator of $T_p$.
We   use
subscript $n$ to emphasize any global variable which may change with respect to $n$.  
$\mathcal L$ represents the   Lebesgue measure in $\mathbb R^d$ and
$B(x,r)$ the closed $d$ dimensional Euclidean ball centered at $x$ with radius $r$ and 
$V_d= \mathcal L(B(0,1))$  the volume of the unit ball $B(0,1)$.   For a vector $x$ we denote with  $\| x\|$ and $\| x\|_{\infty}$ its Euclidean and $L_\infty$  norms, respectively. With a slight abuse of notation, if $f$ is  a real valued function defined over a subset $S$ of $\mathbb{R}^d$, we let $\| f \|_\infty = \sup_{x \in S} | f(x)|$ its $L_\infty$ norm. 
 For any $h>0$ and a measurable set $A\subset \mathbb R^d$ 
we set 
\begin{equation}
	\label{eq:in and out}A_{h} = \bigcup_{x\in A} B(x,h)  \quad \text{and} \quad  A_{-h} = \{x\in A : B(x,h) \subset A\} .
\end{equation}
For any two real sequences $\{a_n\}_{n=1}^\infty$ and $\{b_n\}_{n=1}^\infty$ we write 
$a_n =O(b_n) $ if there exists $C>0$ such that $\lim\sup_{n\to \infty} |a_n / b_n| <C $ and write 
$a_n = \Theta(b_n)  $ if $a_n =O(b_n)$ and $b_n = O(a_n)$. For any two closed subsets $A$ and $B$ of $\mathbb R^n$, we use $d(A, B) = \inf_{x \in A, y \in B} \| x - y \| $ to represent the ordinary distance between them.

\section{Cluster Trees Estimation}

Let $P$ be a probability distribution with a  continuous\footnote{Density based clustering does not require in general continuous densities.}  Lebesgue density
$p$ and with support $\Omega \subset
\mathbb{R}^d$.
	For any $\lambda \geq
	0$, let $  \{ x \in \Omega \colon p(x) \geq \lambda\}$
	be the $\lambda$-upper level set of $p$ and  the
	$\lambda$-cluster of $p$ are  the connected components of $L(\lambda)$. 
See \Cref{section:topology}  for definition of connectedness.
Notice that the set of all
clusters is an indexed collection of subsets of $\Omega$, whereby  each cluster
of $p$ is assigned the index $\lambda$ associated to the corresponding
super-level set $L(\lambda)$, and that many clusters may be indexed by the
same level $\lambda$.
	The cluster
	tree of $p$ is the collection $T_p$ of all clusters of $p$, that is 
	$$ T_p =\{ L(\lambda) \}_{\lambda \ge 0}.$$
	 We can  think of  the cluster tree of $p$ as the function defined on $[0,\infty)$ and 
	for each $\lambda \geq 0$, it returns the set of $\lambda$-clusters of $p$.
Thus,
$T_p(\lambda)$ consists  of disjoint connected subsets of $\Omega$.
 We remark that, since the density $p$ is unique only up to sets of Lebesgue measure zero, the  cluster tree $T_p$ is also not unique. In fact, \cite{steinwart:15} shows that there exists a well-defined notion of  cluster tree for the distribution $P$ that is independent of the choice of the density. Furthermore, if $P$ admits an upper semi-continuous density $p$, then 
the cluster tree is in fact composed of the hierarchy of the (closures of the) upper level sets of such density. As $P$ is assumed, throughout most of the article, to have a density that is continuous everywhere on its support, when we speak of ``the'' density of $P$, we will refer to this canonical choice.

The concept of cluster tree owes its name to the easily verifiable property (see
\cite{hartigan1981consistency})  that if $A$
and $B$ are elements of $T_p$, i.e.  distinct clusters of $p$, then $A \cap B
= \emptyset$ or $A
\subseteq B$ or $B \subseteq A$. This induces a partial order on the set of
clusters. In particular, for any $\lambda_1 \geq
\lambda_2 \geq 0$, if $A \in T_p(\lambda_1)$ and $B \in T_p(\lambda_2)$ then either
$A \cap B = \emptyset$ or $ B \subseteq A$.  As a result, $T_p$ can be represented as a
dendrogram with height indexed by $\lambda \geq 0$. We refer to
\cite{inference.tree} for a formal definition of the dendrogram encoding
a cluster tree.

Let $\{X_i\}_{i=1}^n$ be   i.i.d. samples from $P$.  In order to estimate the cluster tree of
$p$ we will consider tree-valued estimators, defined below.

\begin{definition}
	\label{def.cluster.tree.estimator}
	A cluster tree estimator of $T_p$ is a
	collection  $\widehat{T}_n$ of subsets of $\{X_i\}_{i=1}^n$ indexed by $[0,
	\infty)$ such  that 
\\
$\bullet$		for each $\lambda \geq 0$, $\widehat{T}_n(\lambda)$  consists of disjoint subsets of $\{X_i\}_{i=1}^n$ (including, possibly the empty set), called clusters, and 
		\\
$\bullet$	$\widehat{T}_n$ satisfies the  tree property:
		for any $\lambda_1 \geq \lambda_2 \geq 0$, if $A \in
		\widehat{T}_n(\lambda_1)$ and $B \in \widehat{T}_n(\lambda_2)$ then either $A \cap B
		= \emptyset$ or $A
		\subseteq B$.

\end{definition}

%The first condition above is a regularity condition imposed to avoid trivialities.

It is important to realize that, while the cluster tree $T_p$  is a collection of connected subsets of the support of $p$, the cluster tree estimators considered in this paper are collections of subsets of the sample points.

In order to quantify how well a cluster-tree estimator approximates the true cluster tree, we will make use of  the notion of cluster tree consistency put forward by  \cite{chaudhuri2014consistent}.

	In detail, let $ \mathcal{A}_n $ denote a  collection of
	connected subsets of the support of $p$, which may depend on $n$. A cluster tree estimator $\widehat{T}_n$ is consistent with
	respect to  $\mathcal{A}_n$ if,  with probability tending to $1$ as $n \to \infty$, the following holds simultaneously over all $A$ and $A'$ in
	$\mathcal{A}_n$: 
	the smallest clusters in $\widehat{T}_n$ containing $A \cap \{X_i\}_{i=1}^n$ and $A' \cap \{X_i\}_{i=1}^n$ are disjoint.
The requirement for consistency outlined  above is rather
natural: if a cluster tree is  deemed consistent with respect to the
sequence $\mathcal{A}_n$, then it should, with probability tending to $1$, cluster the sample points perfectly well, 
as if we had  the ability of verifying, for each pair of sample points $X_i$ and $X_j$  and each connected set $A \in \mathcal{A}_n$, whether both $X_i$ and $X_j$ are in $A$.

We  allow $ \mathcal{A}_n$ to grow larger and more complex with $n$, so that the cluster tree estimator
will be able to discriminate among clusters of $p$
that are barely distinguishable given the size of the sample.  An example of a sequence $\{  \mathcal{A}_n \}_{n=1}^\infty$ is the set of $\delta_n$-separated clusters according to \Cref{defi:path-delta-separated}, where the parameter 
$\delta_n$ is taken to be vanishing as $n\to \infty$.
The sequence of target subsets $\{ \mathcal{A}_n \}_{n=1}^\infty $ may not be chosen to be too
large:  for example if $\mathcal{A}_n$ equals to the set of all clusters
of $p$, then, depending on the complexity of $p$, no cluster tree estimator
need to be consistent. 
A natural way to define $\{ \mathcal{A}_n \}_{n=1}^\infty$ is by specifying a {\it separation
	criterion} for sets, which may become less strict as $n$ grows, and then
populate each $\mathcal{A}_n$ using only the connected subsets of the
support of $p$ fulfilling such
a
criterion. 
In particular, \cite{chaudhuri2014consistent} develop a   criterion known as the $(\epsilon,\sigma)$-separation, which
requires two connected subsets  $A$ and $A'$  to be far apart from each other in terms
of their ``horizontal" distance $d(A,B)$
and their ``vertical" distance, in the sense that the smallest cluster containing
both $A$ and $B$ should belong to a level set of $p$ indexed by a value of
$\lambda$ significantly smaller to  the values indexing the level sets of $A$ and
$B$. See \Cref{defi:Dasgupta}  below
for details. One of the main contributions of this paper is to replace this rather general notion of separation by a simpler one, the $\delta$-separation criterion in \Cref{defi:path-delta-separated}, which is better suited deal with smooth densities. This allows us to derive new rates of consistency that depend explicitly on the smoothness of the density. 

As explained in \cite{eldridge2015beyond}, the cluster tree consistency guarantees based on separation criteria can be  fairy coarse, as they only require $\widehat{T}_n$ to preserve the connectivity of all the sets in $\mathcal{A}_n$. In particular, a tree estimator that is consistent with respect to such definition needs not yield a good clustering of  the sample points. Concretely, $\widehat{T}_n$ might have additional unwanted clusters, referred to as {\it false} in \cite{chaudhuri2014consistent}, that do not correspond to any disjoints sets in $\mathcal{A}_n$, a phenomenon referred to as {\it over-segmentation} by \cite{eldridge2015beyond}. Similarly, $\widehat{T}_n$ might not conform to the partial order of inclusions among the clusters of $p$, an issue called {\it improper nesting}. In fact, the  estimators developed in this paper do not suffer from such shortcomings and are consistent in the merge distortion metric of \cite{eldridge2015beyond}, a more refined stronger notion of consistency for cluster trees.  See \Cref{section:consistent split level,section:discussion} below.

\section{The DBSCAN Algorithm}
\label{section:background}
The DBSCAN algorithm, first introduced in \cite{ester1996density}, is an extremely popular methodology for ``flat'' clustering.
In this section we introduce a simple generalization of DBSCAN, shown below in \Cref{algorithm:DBSCAN}, that yields cluster tree estimators and establish its connections with kernel density estimation. 
%Later on in \Cref{algorithm:MDBSCAN} we will propose a variant  that leads to optimal cluster consistency rates for H\"{o}lder differentiable densities. 

\begin{algorithm}[!ht]
	\begin{algorithmic}
		\INPUT i.i.d sample $\{X_i\}_{i=1}^n$,  and $h>0$.
		\State 1. For each $k \in \mathbb{N}$, construct a graph  $\mathbb G_{h,k}$ with nodes $\{X_i: |B(X_i,h)\cap \{X_j\}_{j=1}^n |  \ge k \}$ and edges $(X_i, X_j)$  if  $\| X_i-X_j\| <2 h$. 
		\State 2.  Compute $\mathbb C(h,k)$, the graphical connected components of $\mathbb G_{h,k}$.
		\OUTPUT $  \{ \mathbb C (h,k), k \in \mathbb{N} \}.$
		\caption{The DBSCAN algorithm.}
		\label{algorithm:DBSCAN}
	\end{algorithmic}
\end{algorithm}

For a fixed value of $k$,   \Cref{algorithm:DBSCAN} is in fact a simplified version of the original DBSCAN procedure of  \cite{ester1996density}, where the parameters $h$ and $k$ are called instead $\mathrm{Eps}$ and $\mathrm{MinPts}$, respectively. Notice that, unlike in the original formulation of DBSCAN, we do not distinguish between core and border points and, furthermore, we evaluate connectivity among the sample points using balls of radius $2h$ instead of $h$. Such modifications have no impact on the rates of consistency we obtain but simplify the derivations. 

Assuming $h>0$ fixed, by sweeping through all the possible values of $k$, 
\Cref{algorithm:DBSCAN}  produces a sequence of nested geometric graphs    
$  \widehat T_n= \{  \mathbb C (h,k)\}_{k\in \mathbb{N}}$.
It is immediate to see that  $  \widehat T_n$ forms a cluster tree estimator over the sample points $\{ X_i\}_{i=1}^n$; see \Cref{def.cluster.tree.estimator}.
This is because, for each $k_1 \leq k_2$,
\[
\bigcup_{\{X_i: \ |B(X_i,h)\cap \{X_j\}_{j=1}^n |  \ge k_2 \} }B(X_i,h)\subseteq \bigcup_{ \{X_i: \ |B(X_i,h)\cap \{X_j\}_{j=1}^n |  \ge k_1 \} }B(X_i,h).
\]

	In practice, \Cref{algorithm:DBSCAN} can be efficiently implemented using a union-find structure in such a way that the set $\mathbb{C}(h,k)$ of the maximal connected components of $\mathbb G_{h,k}$ can be computed without using the potentially expensive breadth-first search or depth-first search algorithms. 
   The resulting cluster tree algorithm is simpler  than the estimator based on Wishart's algorithm proposed in
\cite{chaudhuri2014consistent}. Indeed, the DBSCAN-based  estimator is obtained
from a sequence of node-induced sub-graphs of the $2h$-neighborhood graph over the
sample points. In contrast, Wishart's algorithm entails taking node and
edge-induced sub-graphs of the $k$-nearest neighborhood graph over $\{X_1,\ldots,X_n\}$, which
has higher computational complexity.

As explained in \cite{sriperumbudur2012consistency}, DBSCAN is implicitly using
a kernel density estimator with kernel corresponding to the indicator function
of the unit $d$-dimensional Euclidean ball to cluster the points. In detail,
consider the density estimator 
$\widehat{p}_h$ given by 
\begin{equation}\label{eq:ph.dbscan}
x \in \mathbb{R}^d \mapsto  \widehat p_h(x)=\frac{ |B(x,h)\cap \{X_i\}_{i=1}^n |}{nh^dV_d} = \frac{1}{n h^d
	V_d} \sum_{i=1}^n K\left( \frac{x - X_i}{h} \right),
\end{equation}
where 
\begin{equation}\label{eq:KDE}
K(x) = \left\{
\begin{array}{ll}
1 & \text{if } x \in B_d(0,1),\\
0 & \text{otherwise.}    	
\end{array} \right.
\end{equation}
It is easy to see that  $\widehat{p}_h$ is a Lebesgue density, i.e. $\widehat{p}_h(x)$ is a measurable, non-negative function  and $\int_{\mathbb{R}^d}
\widehat{p}_h(x) dx = 1$. Furthermore, 
$
\mathbb{E}[\widehat{p}_h(x)] = p_h(x)$  for all $x \in \mathbb{R}^d,
$
where
\begin{equation}\label{eq:ph}
p_h(x) = \frac{1}{h^d V_d} \int_{\mathbb{R}^d} K\left( \frac{x - z}{h}
\right) p(z) dz  = \frac{P(B(x,h))}{h^d V_d}.
\end{equation} 
For any  $\lambda \geq 0$, set $\widehat D(\lambda)= \{x \colon \widehat p_h(x)\ge
\lambda\}\cap \{X_i\}_{i=1}^n$ and 
\begin{equation}
\label{eq:hat-L}
\widehat L (\lambda) = \bigcup_{X_j\in \widehat D(\lambda) } B(X_j,h).
\end{equation}
Then, setting, for $ k >0$, $\lambda_k=\frac{k}{nh^dV_d}$, one can see that, for clustering purpose, $\mathbb C(h,k)$ and $\widehat L(\lambda_k)$  
convey the same information. 
Indeed from the definition of  $\widehat L(\lambda_k)$, it is straightforward  to see that 
\begin{lemma}
\label{lem:obvious}
 Two data points $X_i$ and $X_j$
are in the same connected component of the $d$-dimensional set $\widehat L(\lambda_k)$ if and only if they are in the same connected component of the graph $\mathbb C(h,k)$.
\end{lemma}

The union of balls $\widehat{L}(\lambda)$ is a renown estimator in the literature on level set
estimation, originally studied in
\cite{DW} (see also \cite{cuevas2004boundary}). In particular,
%Furthermore, as shown for instance in  \cite{sriperumbudur2012consistency}, 
with a suitable
choice of the bandwidth parameter $h$ and as $n$ grows unbounded, $\widehat L (\lambda)
$
is a rate-optimal  estimator  of the level set $L(\lambda)$ under various loss functions and appropriate assumptions on  the
underlying density.

\section{Clustering  Consistency for H\ou lder Continuous Densities}
\label{section:holder}

%!TEX root = ./dbscan_JMLR.tex

In this section we show that the DBSCAN algorithm \ref{algorithm:DBSCAN} is consistent under
H\"{o}lder smooth densities. Towards that end,
we introduce a new notion of cluster tree consistency, called $\delta$-consistency (see \Cref{section:delta consistency} below), which is well-suited to study cluster trees generated by smooth densities. 
We will show that DBSCAN, with suitable inputs, will return  cluster tree estimators that nearly attain the corresponding  minimax optimal rates and that those rates depend  on the degree of
smoothness of the density.

\subsection{H\ou lder smooth densities}
\label{sec:holder.recap}

Below we give a recap of well-known results on non-parametric density estimation.
Given 
vectors $s=(s_1,\ldots, s_d)$ in $\mathbb{N}^d$ and $x
=(x_1,\ldots,x_d)$ in $\mathbb{R}^d$, set
$|s| =s_1+\dots +s_d$ and $x^s= x_1^{s_1}\dots x_d^{s_d}$,
and let
$$ D^s =\frac{\partial ^{s_1+\dots +s_d}}{\partial x_1^{s_1}\dots \partial x_d^{s_d}}$$
denote the differential operator. 
A  function 
$p :\mathbb R^d\to \mathbb R$ is said to belong the H\ou lder class  $\Sigma
(L,\alpha)$ with parameters
$\alpha> 0$ and $L>0$  if $p$ is $\lfloor \alpha \rfloor$-times continuously
differentiable and, for all $x,y \in\mathbb R^d$ and all $s \in
\mathbb{N}^d$ with $|s| = \lfloor \alpha \rfloor$,
$$|D^sp(x)-D^sp(y)| \le L \|x-y \|^{\alpha-s}.$$
Notice that, when $ 0 < \alpha \leq 1$, the H\"{o}lder condition reduces to the
Lipschitz condition
\[
| p(x) - p(y) | \leq L \| x - y\|^\alpha, \quad \forall x, y \in
\mathbb{R}^d.
\]

Let $\widehat p_h$ denote a kernel density estimator with bandwidth $h$ and kernel $K$, that is 
$$\widehat p_h (x) =\frac{1}{n h^d} \sum_{i=1}^n K\left( \frac{x- X_i}{h}  \right). $$
Then, we obtain the standard bias-variance decomposition for the KDEs, namely
\[
\| \widehat{p}_h - p \|_\infty \leq \|\widehat p_h-p_h\|_{\infty}  + \|p_h-p \| _{\infty}.
\]
%where $p$ is assumed to belong to the class $\Sigma(L,\alpha)$. 
In order to control the  stochastic component $ \|\widehat p_h-p_h\|_{\infty}  $  we will invoke well-known concentration bounds from density estimation to conclude that, under appropriate and very mild assumptions on $K$, there exists a constant $C_1 > 0$, depending on $\|p\|_\infty$, $d$ and $K$, such that, for any $\gamma > 0$ and all $n$ large enough and assuming $n h^d \geq 1$, 
\begin{equation}\label{eq:variance of density}
\mathbb{P}\left(\|\widehat p_h-p_h\|_{\infty}\le a_n \right) \ge 1-e^{-\gamma},
\end{equation}
where
\begin{equation}\label{eq:an}
a_n = C_1 \sqrt{\frac{(\gamma+\log (1/ h))}{nh^d} }.
\end{equation}
The verification of this bound can be found in many places in the literature; see, e.g.,  \cite{gine2002rates}, \cite{sriperumbudur2012consistency}, \cite{jiang2017uniform} and \cite{jisu.kde}. See  \Cref{appendixa} for details.  
As for the bias term $\| p_h - p \|_\infty$, if $K$ is chosen to be a 
 $ \alpha  $-valid kernel\footnote{ 
For a fixed $\alpha>0$, a function $K :\mathbb R^d\to  \mathbb R$ is an $\alpha$-valid kernel if $\int_{\mathbb R^d} K(x) dx =1$, has finite $L_p$ norm for all $p \geq 1$, $\int_{\mathbb R^d} \| x\|^\alpha K(x) dx < \infty$ and $\int_{\mathbb R^d}  x^sK(x) dx=0$ for all $s = (s_1,\ldots,s_d) \in \mathbb{Z}^d$ such that $1 \leq \sum_{i=1}^d s_i \le \lfloor \alpha \rfloor $, where for $x = (x_1,\ldots,x_d) \in \mathbb{R}^d$, $x^s = \prod_{i=1}^d x_i^{s_i}$. See Definition 1 in \cite{rigollet2009optimal}.
 } (see, e.g., \cite{rigollet2009optimal}), then
standard calculations yield that
\begin{equation}\label{eq:ph.bias}
\| p_h - p \|_\infty \leq C_2 h^\alpha,
\end{equation}
for an appropriate constant $C_2>0$ depending on $L$ and $\alpha$. In particular, since this type of kernels are polynomials 
supported on $[0,1]^d $, they automatically satisfy the VC condition \citep[see lemma 22 of][for a justification]{nolan1987u}.

Thus combining the bias in  \eqref{eq:variance of density} and the variance in  \eqref{eq:ph.bias}, we  conclude that,  with probability at least $1-e^{-\gamma}$, with $\gamma$  any positive number,
\begin{equation}
\label{eq:KDE risk}
\|\widehat p_h- p \|_{\infty}\le C_1 \sqrt{\frac{\gamma+\log (1/ h)}{nh^d}} +C_2 h^{\alpha}.
\end{equation}
Setting, e.g., $\gamma = \log n$, the optimal choice of the bandwidth is 
\begin{equation}\label{eq:h.optimal}
h \asymp \left( \frac{\log n}{n} \right)^{2\alpha + d},
\end{equation}
leading to the final rate of
$
  \left( \frac{\log n}{n} \right)^{\frac{\alpha}{2\alpha + d}} 
$, which is in fact minimax optimal.

\subsection { The $\delta$-Separation Criterion}
\label{section:delta consistency}
We now formulate a notion of cluster separation that is
naturally suited to smooth densities and, for continuous densities  is equivalent to separation in the merge distance of \cite{eldridge2015beyond} \cite[see also][]{inference.tree}.

\begin{definition} \label{defi:path-delta-separated}

Let  $A$ and $A'$ be subsets  the support of the density $p$ and set $\lambda :=\inf_{x\in A\cup A'}p(x)$.  For $\delta \in (0,\lambda)$,  $A$ and $A'$ are said to be $\delta$-separated  if
	they 
	belong to distinct connected components of the level set $\{x \colon p(x) >\lambda -\delta\} $.
\end{definition}

Unlike the separation criterion of \cite{chaudhuri2014consistent}, which requires the specification of two parameters quantifying the horizontal and vertical displacement between clusters, $\delta$-separation only uses one parameter
The intuition behind the notion of $\delta$-separation is simple: due to the smoothness of the
density,  the degrees of ``vertical" and
``horizontal" separation between clusters are coupled. This is illustrated in 
\Cref{fig:projection_pairs_plot} and best explained for the case of a density in
$\Sigma(L,\alpha)$ with $\alpha \leq
1$. If $A$ and $A'$ are $\delta$-separated, then their distance is at
least $\left( \frac{\delta}{L} \right)^{1/\alpha}$. 
As a result, the degree of separation between clusters of smooth densities  can
be described using only
one parameter, a  feature that we will exploit to derive new notion of consistency for
clustering.

\begin{definition}\label{defi:delta-consistent}
Let $\delta>0$ and $\gamma \in (0,1)$.  A cluster tree estimator based on an
i.i.d. sample $\{X_i\}_{i=1}^n$ is
$(\delta,\gamma)$-accurate if, with probability no smaller than $1-\gamma$, for any 
  pair of connected subsets $A$ and $A'$ of the support that are $\delta$-separated, exactly one of the following conditions holds:
\begin{itemize}
\item[1.] at least one of  $A\cap \{X_i\}_{i=1}^n$ and $A'\cap \{X_i\}_{i=1}^n$ is empty;
\item[2.] the smallest clusters in the cluster tree estimator containing $A\cap \{X_i\}_{i=1}^n$ and $A'\cap \{X_i\}_{i=1}^n$ are disjoint.
\end{itemize}
Let  $\{ \delta_n \}_n$ be  a vanishing sequence of positive numbers and a $\{
\gamma_n \}$ a vanishing sequence
in $(0,1)$. We say that the  sequence of cluster tree
estimators $\{ T_n \}_n$, where $T_n$ is based on an i.i.d.
sample $\{X_i\}_{i=1}^n$, is 
$\delta$-consistent with rate $\delta_n$ if, for all $n$ large enough,
$T_n$ is $(\delta_n,\gamma_n)$-accurate where $\gamma_n$ decays polynomially in $n$.
%	A cluster tree estimator based on the data $\{X_i\}_{i=1}^n$ is
%	$\delta$-consistent if  for any 
%	pair of connected subsets $A$ and $A'$   that are $\delta$-separated, exactly one of the following conditions holds:
%	\\
%	1. at least one of  $A\cap\{X_i\}_{i=1}^n$ and $A'\cap\{X_i\}_{i=1}^n$ is empty;
%	\\
%	2. the smallest clusters in the cluster tree estimator containing $A\cap\{X_i\}_{i=1}^n$ and $A'\cap\{X_i\}_{i=1}^n$ are disjoint.
\end{definition}

It is important to realize that the notion of
$\delta$-consistency is a {\it uniform} notion of consistency
that is required to hold simultaneously over all possibly pairs of
$\delta$-separated connected subsets of the support.

The $\delta$-separation criterion is closely related to the concept of the {\it merge height} introduced by \cite{eldridge2015beyond}. In the context of hierarchical clustering, the merge height is used to describe the “height” at which two points or two clusters merge into one cluster;   see \Cref{defi:merge height}. 
 In particular we show below in \Cref{lemma:merge height} that if two subsets $A$ and $A'$ of the support are $\delta$-separated and  
	$\inf_{x\in A\cup A'}p(x) =\lambda $, then their merge height  is no larger than $\lambda -\delta$. To further emphasize how similar the two approaches are, we mention that our results about cluster consistency still hold for a slightly stronger
notion of cluster consistency, whereby condition 2 in
\Cref{defi:delta-consistent} is replaced by the  condition 
\begin{enumerate}
\item[2.] there exists a level $\lambda \in [\inf_{x\in A\cup A'} f(x)-\delta,
	\inf_{x\in A\cup A'} f(x))$ such that  $A\cap \{X_i\}_{i=1}^n$ and $A'\cap \{X_i\}_{i=1}^n$  are
	contained in  two different $\lambda$-clusters of the cluster tree
	estimator.
\end{enumerate}
The key difference between $\delta$-consistency and this stronger version of $\delta$-consistency is
that in the latter case we further constrain the split level for $A$ and $A'$ in the cluster tree
estimator to occur at a value less than $\inf_{x\in A\cup A'} f(x)$ by an amount no larger than $\delta_n$. This is precisely what is required for merge distance consistency; see \cite{eldridge2015beyond}.
	We provide more detailed comparison in \Cref{section:discussion}, where we further elucidate the differences between our notion of $\delta$-separation and the $(\epsilon ,\sigma)$-separation criterion of \cite{chaudhuri2014consistent}.

%We also  remark that the notion of $\delta$-consistency is a parametrized version of the   Hartigan consistency introduced in \cite{hartigan1981consistency}. 
% In \Cref{section:discussion},we include  the definition of Hartigan consistency and the definition  of 
%-consistency introduced in \cite{chaudhuri2014consistent} and the detailed comparison with the $\delta$-consistency.

\begin{figure}[H]
	\includegraphics[width=\columnwidth]{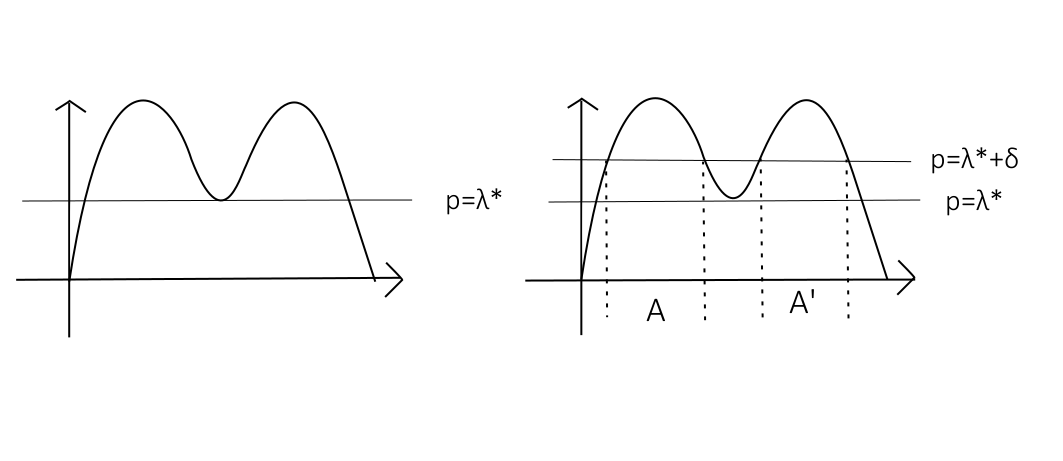}
	\caption{The left figure depicts a split level (defined in $\lambda^*$ \Cref{subsection:split level})   of the density $p$. The right figure depicts two sets $A$ and $A'$ being $\delta$ separated with respect to  $\lambda^*$.}
	\label{fig:projection_pairs_plot}
\end{figure}

\subsection{The split levels}

\label{subsection:split level}
One of the most impoertant features of a cluster tree is the collections of
levels $\lambda$ at which the clusters split into two or more disjoint sub-clusters, which we  refer to as {\it split levels.}
Such levels belong to the well-known class of ``critical levels''  in  differential topology, which identify critical changes in the topology of the upper level sets of $p$. See, for example, \cite{hirsch2012differential} for more details.  
In particular, the estimation of split levels is a central theme in the contributions of  \cite{sriperumbudur2012consistency} and
in \cite{steinwart:15}.
Below, we provide a slightly different characterization  of the split levels of continuous densities and relate it to the
criterion of $\delta$-separation of clusters. The notion of split levels will be important below in \Cref{sec:MDBSCAN upper} in formalizing  conditions under which computationally efficient and statistically optimal cluster tree estimation is feasible for H\"{o}lder densities with smoothness degree $\alpha>1$. It will also be used to demonstrate that we can easily remove  false clusters from the cluster tree estimations returned by our algorithms (see \Cref{section:consistent split level}).

\begin{definition}\label{defi:split-level}
	Let $p:\mathbb R^d \to \mathbb R$ be  a continuous density function. For a fixed $\lambda^* >0$, let $\{\mathcal C_k\}_{k=1}^K$ be the connected components of $\{ p\ge \lambda^*\}$. 
	The value $\lambda^*$ is said to be a split level of $p$ if there exists a $\mathcal C_k$ such that $\mathcal C_k \cap \{p > \lambda^*\}$ has two or more connected components.
\end{definition}

The following simple result illustrates the main topological properties of split levels.
\begin{proposition} \label{lemma:well-defined-2}
	Suppose that $p:\mathbb R^d \to \mathbb R$ is compactly supported and that  $A$ and 
	$A'$ are subsets of  two distinct connected components of $\{p\ge\lambda_1\}$. If $A$ and $A'$ belongs to the same connected components of $ \{ p\ge\lambda_2\}$, 
	where $\lambda_2< \lambda_1$,  then  there is a unique split level $\lambda^{*}
	\in [\lambda_2, \lambda_1)$
	such that $A$ and $A'$ belong to one
	connected component of $\{ p\ge \lambda^*\}$ and to distinct connected
	components of $\{p>\lambda^*\}$.
\end{proposition}

 \Cref{lemma:well-defined-2} suggests that if two connected components
merge into one as the density level $\lambda$ decreases, then there exists  one
and only one split level at which the corresponding merge takes place. 
Therefore,  the following definition, which characterizes the  split level of any two distinct clusters in a cluster tree, seems natural.

\begin{definition}\label{definition:split at a level}
	Suppose $A$ and $A'$ are two open subsets of the support of $p$. Then $A$ and $A'$ are said to split at level $\lambda^*$ if $A$ and $A'$ belong to one connected component of  $\{ p\ge
	\lambda^*\}$ and to two distinct connected components of $\{p>\lambda^*\}$.
\end{definition}

In our next result we illustrate a direct link between the notion of split levels and the criterion of $\delta$-separation introduced above. We will exploit this fact  later  in \Cref{section:consistent split level} to demonstrate how to prune the cluster tree estimators to  yield accurate estimates of the split levels without producing false clusters.

\begin{corollary}
	\label{coro:delta split}
	Let  $A$ and $A'$ be $\delta$-separated. Then there exists a split level 
	$\lambda^* $ of the density, with 
	\[
	\lambda^*\le \inf_{x\in A\cup A'} f(x) -\delta,
	\]
	such that $A$ and $A'$ belong to one connected component of  $\{ p\ge
	\lambda^*\}$ and to two distinct connected components of $\{p>\lambda^*\}$.
\end{corollary} 

\subsection{Rate of consistency for the DBSCAN algorithm} 
\label{section:upper bounds}

We are now ready to present the main results of the paper, and derive rates of consistency for DBSCAN-based cluster tree estimators with respect to the criterion of $\delta$-separation and for H\"{o}lder smooth densities. Specifically,   we will show that these estimators are $\delta$-consistent with rate
\begin{equation}\label{delta.rate}
\delta_n \geq C \left( \frac{\log (n)}{n} \right)^{ \frac{\alpha}{2\alpha + d}},
\end{equation}
for an appropriate constant $C$ that depends on $\| p\|_\infty$, $L$, $K$ and $\alpha$.
The above rates  depend
on the smoothness of the underlying density, with smoother densities leading to faster rates, and, as shown in \Cref{subsection:flat-lowerbound}, are in fact minimax optimal. This is one of the main contributions of this article and delivers an extension of the cluster consistency results of \cite{chaudhuri2014consistent}, which are agnostic to the smoothness of $p$.

\subsubsection{Consistency for $\alpha \leq 1$}

We first show that, when $\alpha \leq 1$, the DBSCAN algorithm is $\delta$-consistent with rate of order \eqref{delta.rate}. We remark that this type of result can be deduced from several  contributions in the literature on density-based clustering, which show that variants of
the DBSCAN algorithm  lead to some  form of cluster consistency when $\alpha \le 1$. See, e.g., \cite{rinaldo2010generalized}, \cite{sriperumbudur2012consistency}, \cite{jiang2017density} and  \cite{ingo.bharath.new}.  We provide the details for completeness.
%For completeness, we summarize these  existing results in the context of  $\delta$-consistency. % , but  we do not claim any originality of the following result.

In order to demonstrate that DBSCAN is
$\delta$-consistent, it will be sufficient to show that the procedure provides
an approximation to the upper level sets of $\widehat{p}_h$. This is done in the next result, which relies on general, well-known, finite sample concentration bounds for KDEs along with standard calculations for the bias of a KDE; see \Cref{lemma:standard-level-inequality} and \Cref{prop:Steinward} in \Cref{appendixa}. %\Cref{section:KDE}.

%\In order to demonstrate that DBSCAN is
%$\delta$-consistent when $\alpha \le 1$ as in \Cref{coro:DBSCAN}, it will be sufficient to show that the following lemma.
\begin{lemma}\label{lemma:DBSCAN-level-inequality}
	Assume that $p \in \Sigma(L,\alpha)$, where $\alpha \in (0,1]$, and let $K$ be
	the spherical kernel. Then, there exist constants  $C_2$, depending on $C_1$,
	$\|p\|_{\infty}$,  $L$ and $d$   such that, 	if  $h =C_1 n^{-\frac{1}{2\alpha +d }}$
	then uniformly over all $\lambda>0$,  with probability at least $1-1/n$,

	\begin{align}\label{eq:dbscan level set}
	\left\{ p\ge \lambda +C_2 \left( \frac{\log(n)}{n} \right)^{\alpha/(2\alpha+d)}+Lh^\alpha \right
	\}  \subset \bigcup_{X_j\in \widehat D(\lambda)}  B(X_j,h)\subset \left\{ p\ge
	\lambda -C_2 \left( \frac{\log(n)}{n} \right)^{\alpha/(2\alpha+d)}-Lh^\alpha   \right \}.
	\end{align}

\end{lemma}

%\Cref{coro:DBSCAN} is  a direct consequence  of \Cref{lemma:DBSCAN-level-inequality}, and  we see that the DBSCAN algorithm,
%with an appropriate choice of $h$, outputs a $\delta$-consistent cluster tree
%with the consistency rates adaptive to $\alpha$. 

As a direct corollary, we see that the DBSCAN algorithm,
with an appropriate choice of the bandwidth $h$, outputs a $\delta$-consistent cluster tree
with consistency rates that depend on  $\alpha$. 

\begin{corollary}\label{coro:DBSCAN}
Assume that $p \in \Sigma(L,\alpha)$, where $\alpha \in (0,1]$, and let $K$ be
	the spherical kernel. Then, there exist constants  $C_1$ depending on
	$\|p\|_{\infty}$,  $L$ and $d$   such that, 	if  $h =C_1 \left( \frac{\log(n)}{n} \right)^{\frac{1}{2\alpha +d }}$, 
	 	the cluster tree returned by the DBSCAN \Cref{algorithm:DBSCAN} is 
	$\delta$-consistent with rate
	$\delta_n \ge C \left( \frac{\log (n)}{n} \right) n^{ \frac{\alpha}{2\alpha + d}} $, 
	where $C= C(\|p\|_{\infty}, L,d)$.
\end{corollary}

\subsubsection{Consistency for $\alpha >1$}
\label{sec:MDBSCAN upper}
When $\alpha>1$,  \Cref{algorithm:DBSCAN} no longer delivers the optimal rate displayed in \eqref{delta.rate}, for  two  reasons. The first  reason stems from standard non-parametric density estimation considerations: when $\alpha > 1$ it becomes necessary to rely on  smoother kernels, namely $\alpha $-valid kernels as indicated before. This will lead to a bias $\| p - p_h \|_\infty$  of the correct order $O(h^\alpha)$.  
The second reason is more subtle: the straightforward arguments we used to handle the case of $\alpha \leq 1$ do not lead to optimal clustering rates  even if the kernel $K$ is chosen to be $\alpha$-valid. To exemplify, suppose we would like to cluster the sample points $\{X_i\}_{i=1}^n \cap \{ x \colon \widehat{p}_h(x) \ge \lambda \}$  for some $\lambda > 0$. The computationally efficient linkage rule implemented by DBSCAN  is to cluster  the points based on the connected components of the union-of-balls around them, i.e based on the connected components of
\[
\widehat{L}(\lambda) = \bigcup_{X_j \in  \{  \widehat{p}_h \ge \lambda \}}   B(X_{j},h ).
\]
Assume now that the gradient of $p$ has norm uniformly bounded by a constant $D$ for all $x \in L(\lambda)$. Then, 
\begin{equation}\label{eq:none vanishing gradient}
\max_{X_j \in  \{  \widehat{p}_h \ge \lambda \} } \sup_{x \in B(X_{j}, h)} |p(x) - p(X_{j})|  \leq D h,
\end{equation}
and, as a result, 
\begin{align}\label{eq:estimate general alpha}\left\{ p\ge \lambda +C \left( \sqrt{\frac{\log(n)}{nh^d}}+ h^\alpha \right) + D h \right
\}  \subset \widehat L(\lambda) \subset \left\{ p\ge
\lambda -C \left( \sqrt{\frac{\log(n)}{nh^d}} - h^\alpha \right) -Dh   \right \},
\end{align}
where   $ C \left( \frac{\log(n)}{{\sqrt {nh^d}}  } + h^\alpha \right) $ comes from  the  $L_{\infty}$ error bound   of $ \alpha  $-valid kernels as in \eqref{eq:KDE risk} and $D h$ is due to  \eqref{eq:none vanishing gradient}. As $h \rightarrow 0$, the term $D h$ dominates the bias term $C h^{\alpha}$, so that the optimal choice of $h$ is of the order $h \asymp \left(\frac{\log n}{n} \right)^{1/(2 + d)}$, which in turn yields a worse rate than \eqref{delta.rate} when $\alpha>1$. What is more, $\| \nabla p(X_j)\|>0  $ for any $X_j$  away from the critical points of $p $. This would mean that
\[
\min_{X_j \in  \{  \widehat{p}_h \ge \lambda \} } \| \nabla p(X_{j}) \| >0 ,
\] 
Then, as $h \rightarrow 0$,  $\sup_{x \in B(X_{j}, h)} |p(x) - p(X_{j})|  \approx \| \nabla p(X_{j}) \| h  =  \Theta(h)$. Thus, the inclusions in \eqref{eq:estimate general alpha} are tight, showing that the sub-optimal choice of $h$ cannot be ruled out. 
We believe that this phenomenon is not specific to DBSCAN only, but applies more broadly to the the class of single-linkage-type clustering algorithms. That is, it seems to us that such algorithms are in general unable, like our DBSCAN-based procedure in \Cref{algorithm:DBSCAN}, to take advantage of a higher degree of smoothness of the underlying density.

The issue outlined above can be handled in more than one way. A possible solution, which is nearly trivial but impractical, is to deploy a computationally inefficient algorithm that assumes the ability to evaluate the connected components of the upper level set of $\widehat{p}_h$ exactly: see \Cref{algorithm:clustering cc} in the appendix. It is immediate to see that this approach produces optimal $\delta$-consistency; see \Cref{coro:flat-upper} in \Cref{algorithm:clustering cc}. Unfortunately, this procedure will require evaluating $\widehat{p}_h$ on a fine grid, which is computationally infeasible even in small dimensions. The second, more interesting and novel solution which we describe next, is to further assume that $p$ satisfies additional mild regularity conditions around the split levels. The conditions are of geometric and analytic nature and are reminiscent of low-noise type assumptions in classification. Under these conditions, the modified DBSCAN \Cref{algorithm:MDBSCAN} will achieve the optimal rate \eqref{delta.rate} while remaining computationally efficient. This finding, stated formally in \Cref{prop:MDBSCAN} below,  is the main result of this section.

\begin{algorithm}[!h]

	\begin{algorithmic}
		\INPUT i.i.d sample $\{X_i\}_{i=1}^n$,  a  $ \alpha $-valid  kernel $K$ and $h>0$.
		\State 1. Compute $\{ \widehat p_h(X_i),  i = 1,\ldots,n \}$.
		\State 2.  For each $\lambda \geq 0 $,  construct a graph  $\mathbb G_{h,\lambda}$ with node set
		$$\widehat D(\lambda)=\{X_i: \widehat p_h(X_i) \ge \lambda\}$$ and edge set $\{ (X_i, X_j) :  X_i,X_j \in \widehat{D}(\lambda) \; \text{and} \; \| X_i-X_j\|< 2h\}$.  
		\State 3. Compute $\mathbb C(h,\lambda)$, the graphical connected components of $\mathbb G_{h,\lambda}$.
		\OUTPUT $\widehat{T}_n = \{ \mathbb C(h,\lambda) , \lambda \geq 0 \}$.
		\caption{The modified DBSCAN}
		\label{algorithm:MDBSCAN}
	\end{algorithmic}
\end{algorithm}

\begin{remark}
	Despite its seemingly different form, \Cref{algorithm:MDBSCAN} is nearly identical to \Cref{algorithm:DBSCAN}. The only difference is the use of an $ \alpha  $-valid  kernel $K$ instead of a spherical kernel. Furthermore, the procedures only require evaluating at most $n+1$ different graphs: 
	\[
	\mathbb G_{h,0}, \mathbb G_{h,\widehat{p}_h(X_{\sigma_1})}, \ldots,  \mathbb G_{h,\widehat{p}_h(X_{\sigma_n})},
	\] 
	where $(\sigma_1,\ldots,\sigma_n)$ is a permutation of $(1,\ldots,n)$ such that 
	\[
	\widehat{p}_h(X_{\sigma_1}) \leq \widehat{p}_h(X_{\sigma_2}) \leq \ldots \widehat{p}_h(X_{\sigma_n})
	\]
	And, again just like with \Cref{algorithm:DBSCAN}, the connected components of each $\mathbb C(h,\lambda)$ can be easily evaluated by maintaining a union-find structure.
\end{remark}

To formulate the the extra regularity conditions  on the geometry of the density $p \in \Sigma(\alpha,L)$ around the split levels that guarantee optimality of the clustering \Cref{algorithm:MDBSCAN} we first recall some notions commonly used in the literature on level set and support estimation. Below, $\Omega$  denotes a generic subset of $ \mathbb{R}^d$ of dimension $d$.\\
\\
%We begin by formulating two widely used technical conditions on any generic manifold $\Omega \subset \mathbb{R}^d$.\\
{\bf C1.} (The Inner Cone Condition) The subset $\Omega$ satisfies the inner cone conditions if
ihere exist
constants 
$r_I,c_I>0$  such that, 
for any $0\le r\le r_I$ and 
$x\in \Omega $,
$$\mathcal L(B(x,r)\cap  \Omega) \ge c_IV_d r^d,$$
where   $\mathcal L $   denotes the Lebesgue measure of $\mathbb R^d$.\\ 
\\
{\bf C2.} (The Covering Condition) The subset $\Omega$ satisfies the covering condition 
if there exists a constant $C_I$ such that, 
for any $0< r\le r_I$, there exists a collection of points $\mathcal N_r\subset \Omega$ such that $card(\mathcal N_r)\le C_I r^{-d}$ 
and 
$$\bigcup_{y\in\mathcal N_r} B(y,r) \supset \Omega.$$

Both assumptions {\bf C1, C2}  are rather mild. If $\Omega$ is a compact manifold of dimension $b \leq d$ with piecewise Lipschitz boundary, both assumptions are automatically verified with the dimension $d$ replaced by the intrinsic dimension $b$.  
\citep[see, e.g.][]{do1992riemannian}. 
The inner cone condition {\bf
	C1} is used  in  \cite{korostelev2012minimax} and is well-known as the standard
condition \citep{cuevas2009set} or, more recently, the $(a,b)$ condition of \citep{Chazal:2015:CRP:2789272.2912112}. It is essentially equivalent to the level set regularity condition [B] in \cite{hausdorff}. The covering condition {\bf C2}  holds automatically if $\Omega$ is compact. See, e.g.,  \cite{rinaldo2010generalized} and  \cite{balakrishnan2012}.

Since  $p\in \Sigma(L,\alpha)$ with $\alpha > 1$, 
any level set  $\{p\ge \lambda\}$ is a union of connected $d$ dimensional manifolds with $C^1$ boundary.
Therefore it is natural  to require both {\bf C1} and {\bf C2} to hold simultaneously for all the upper level-sets of $p$ {\it right above the split levels.} Specifically, we will assume the following.  
\\
\\
{\bf C.}
There exists  a $\delta_0>0$ such that,
for any split level  $\lambda^* $ of $p$ and any $0<\delta \le \delta_0$, the set
$ \{x \colon p(x) \ge \lambda^*+\delta\} $ satisfies conditions {\bf C1} and {\bf C2} with
 constants $r_I$,$c_I$ and $C_I$ only depending on $p$. 
\\
\\
We also need the connected components of the upper level sets right above the split levels to satisfy a low-noise condition as follows. 
\\
\\
{\bf S($\alpha$)}.
There exist positive constants $\delta_S$ and $c_S$ such that, for each split level $\lambda^*$ of the density $p$, the following holds.
Let $\{\mathcal C_k\}_{k=1}^K$ be the connected components of $\{x \colon p(x) >\lambda^*\}$.
Then,
\begin{equation}\label{eq: holder_seperation_condition}
\min_{k\not = k'}d(\mathcal C_k \cap \{p\ge \lambda^*+\delta \} , \mathcal C_{k'} \cap \{p\ge \lambda^*+\delta \} ) \ge c_S \delta^{1/\alpha}, \quad \forall \delta \in (0, \delta_S].
\end{equation}

Condition {\bf S($\alpha$)} constrains the behavior of the density only around the split levels. It is a fairly common assumption in the literature: it coincides with the {\it separation exponent} condition of \cite{steinwart:15} (see Definition 4.2 therein), which quantifies the separation of distinct connected components right above the split levels. Furthermore, {\bf S($\alpha$)} is implied by {\it the local density regularity} conditions of \cite{hausdorff}, which in turn is used in \cite{jiang2017density} to define the {\it $\beta$-regularity}  condition for cluster separation. %used by  for any level $\lambda$ that is not a critical level, which in turn 
\cite{steinwart:15} provides several specific examples of densities satisfying the {\bf S($\alpha$)} condition. In fact, we prove that conditions {\bf C} and {\bf S($\alpha$)} are verified in a  large non-parametric class  of functions. This class consists of  Morse density functions, which are widely used in
the density based clustering and mode estimation and topological data analysis;
see, e.g., \cite{chacon2015population}, \cite{arias2016estimation} and
references therein. We recall that a function $p$ is Morse if all its critical points have a non-degenerate Hessian. An equivalent and more intuitive condition is that $p$ behaves like a quadratic function around its critical points.

\begin{proposition}
 \label{lemma:properties of morse}
 	Suppose $p : \mathbb R^d \to \mathbb R $ is a Morse function. Then $ p$ satisfies {\bf C} and {\bf S(2)}.
\end{proposition}

Another interesting class of density functions satisfying conditions {\bf C} and {\bf S($\alpha$)} can be obtained as follows.
Let $\alpha\ge 2$ be any integer and $f_1: [0,1]\to \mathbb R$ be such that 
$f_1(x) =(x-2)^\alpha$. Then, there exists a polynomial $f_2$ of degree $\alpha$ such that the function on $\mathbb{R}$ defined point-wise as
$$f(x) =
\begin{cases}f_1(x), \quad x\in [1,2]
\\
f_2(x), \quad x\in [0,1]
\\
0, \text{ otherwise,}
\end{cases} $$ 
has continuous  derivatives up to order $\alpha-1$  and is such that $f(0)=f'(0)= ,\ldots, =f^{(\alpha-1)}=0 $.  When $\alpha=3$, $f$ is a natural spline.
For any integer $d \geq 1$, let $ F :\mathbb R^d \to \mathbb R$ be such that $F(x) = f(\|x\|_2) $.  Then $ F \in \Sigma(\alpha, L)$.
Denote $x_0 = (2,0,\ldots,0) $.
Let $G (x) = F(x-x_0) +F(x+x_0) $. It is easy to see that  for any $ 0<\delta\le 1$
$$ \{ G(x )\ge \delta  \} = B(x_0, 2-\delta^{1/\alpha}) \cup B(-x_0, 2-\delta^{1/\alpha}) .$$
As a result, conditions {\bf C} and {\bf S($\alpha$)} are trivially satisfied in this simple case.

Our main result of this section is to prove that that the conclusion of \Cref{coro:DBSCAN} still holds for $\alpha>1$, provided that the conditions {\bf C} and {\bf S($\alpha$)}  are met.

\begin{theorem}
	\label{prop:MDBSCAN}
	Let $p\in \Sigma(\alpha>1,L)$ be any density function with compact and connected support and
	finitely many split levels.   Suppose that conditions {\bf C} and {\bf S($\alpha$)}
	hold for $p$.	If 
	$h \asymp \left( \frac{ \log(n)}{n} \right)^{1/(2\alpha+d) })$,  then, with probability at least   $ 1-\frac{1}{n} - O( h^{-d} \exp (-c  n^{\alpha/(2\alpha+d)})) $,  the cluster tree returned by the modified DBSCAN \Cref{algorithm:MDBSCAN} is $\delta$-consistent with rate
		$$ \delta_n \geq 2a_n +(4h/c_S)^{\alpha}  $$
	where $c$ is a constant that depends on $p$ only,
 $a_n  =C_1 \sqrt{\frac{(\log n+\log (1/ h))}{nh^d}} +C_2 h^{\alpha}$ is the right hand side of the inequality in \eqref{eq:KDE risk} and $c_S$ is defined in {\bf C}. 

\end{theorem}
The choice of the parameter $h$ in \Cref{prop:MDBSCAN} yields that \Cref{algorithm:MDBSCAN} is $\delta$ consistent with rate given by \eqref{delta.rate}.

\subsection{Lower bounds}

\label{subsection:flat-lowerbound}
Next, we show that the consistent rates  of the DBSCAN
algorithm derived in the previous sections
are nearly minimax optimal, save for a $\log(n)$ term.
%We emphasize that in our minimax analysis we do not directly 
%the lower bound construction for $L_\infty$ estimation of H\"{o}lder
%smoother  density estimation,  as our estimators are   cluster trees  and not  functions in $\mathbb R^n$.  
We point out that the lower bound results by
\cite{chaudhuri2014consistent} are not directly applicable to our problem, since they rely on discontinuous densities.  %Instead, our arguments are inspired by the lower bound construction for density for level set estimation developed in \cite{rigollet2009optimal}.

\begin{theorem} \label{lemma:Holder-lower-bound} Suppose  $d\ge 1$ and  $\alpha
	> 0$.
	There exists a finite family  $\mathcal{F}$   of $d$-dimensional probability density functions 
	belonging to the H\"{o}lder class  $\Sigma (L,\alpha)$  satisfying the conditions {\bf C},  {\bf S($\alpha$)} and uniformly
	bounded from above by $C_0$,  and a constant
	$\K$, depending on $L$ and $\alpha$, such that  when 
	\[
	n\ge\frac{4^d8\log(32)}{V_d} \quad \text{and} \quad 
	\delta\le \min \left\{\left( \frac{\mathcal
		K}{16^{\alpha}(7C_0)^{\alpha/d} }\right) ,\|p\|_{\infty}/(2^{d/2+1})\right\},
	\]
	where $V_d$ denote the volume of a $d$ dimensional ball, the following holds. 
	If cluster tree estimator  is
	$(\delta_n,1/4)$-accurate
	when presented with an i.i.d. sample from a density function in $\mathcal{F}$,
	then it must be the case that
	\begin{equation}\label{eq:flat-lower-1}
	n\ge \frac{C_0 \K^{d/\alpha}}{C\delta_n^{2+d/\alpha}},
	\end{equation} 
	for some constant $C$ only depends on $d$.
\end{theorem}
  Therefore,  
	with the constant $C_0$ in  the previous theorem and the dimension $d$ fixed,
the bounds obtained in \Cref{prop:MDBSCAN}  and \Cref{coro:DBSCAN} match the  minimax bound in \eqref{eq:flat-lower-1}, up to a $\log(n)$ factor. 
	Thus, together they show that, up to log factors,  the optimal rate for
	$\delta$-consistency of density functions in $\Sigma(L,\alpha)$ is  of order $\left(\frac{\log(n)}{n}\right)^{\alpha/(2\alpha+d)}$. 

Interestingly, the minimax clustering rates we derived match the rates for estimation of a H\"{o}lder density $p$ under $L_\infty$ norm; see \Cref{sec:holder.recap}. While this result may not be entirely surprising in light of the findings of, e.g., \cite{eldridge2015beyond} and \cite{inference.tree}, such a a connection has never been formally established, to the best of our knowledge. In particular, our results seem to settle, at least for the class of  H\"{o}lder-continuous densities and with respect to the criterion of $\delta$-consistency, a long-standing open problem of how  density-based clustering compares to density estimation: both problems exhibit the same degree of statistical difficulty.

	%Interestingly this rate matches
	%the minimax rate for
	%estimating $\alpha$-smooth densities in the $L_\infty$ norm. 
	%This fact is perhaps not very surprising. 
	% However, to the best of our knowledge, the  matching lower  bounds have  not been previously 
	%established in the literature in a rigorous manner.

\subsection{Consistent Estimate of The Split Levels}
\label{section:consistent split level}
In this section, we present a simple pruning strategy, leading to consistent
estimators of the split levels of the the density. While 
pruning strategies and consistent estimation of split levels have been considered by several authors, such as \cite{sriperumbudur2012consistency,steinwart:15}, \cite{chaudhuri2014consistent} and \cite{jiang2017density}, the existing results do not yield error bounds that depend on the degree of smoothness $\alpha$ for density
$p \in \Sigma(L,\alpha)$ with $\alpha>1$. %Consequently, in this section, we will identify significant split level from the cluster tree returned by the modified DBSCAN algorithm when $p \in \Sigma(\alpha,L)$ with $\alpha>1$. 
\\
\\
The following definition provides  a way to identify  significant  split levels in
the cluster tree estimator returned by \Cref{algorithm:MDBSCAN}.
\begin{definition}\label{definition:pruning} Let   $\Delta >0$. 
	The random variable $\widehat {\lambda^*} \in (0,\infty)$ is said to be  a $\Delta$-significant split level of
	the   cluster tree estimator if there exist two data points $X_i,X_j\in  D (\widehat {\lambda^*} +\Delta)$ such that   
	\begin{equation}
	\widehat {\lambda^*} =\sup\{\lambda>0: X_i \text{ and } X_j \text{ are in the same
		connected component of }
	\mathbb C (h, \lambda )   .  
	\}
	\label{eq:find-split-hat}
	\end{equation}
\end{definition}
%The motivation behind the definition of $\Delta$-significant split levels is that in theory,
%the accuracy of modified DBSCAN estimator is limited with finitely many data points.
%By looking at split levels corresponding to   large clusters, 
%we rule out the insignificant split levels and only keep the $\Delta$-significant ones. Therefore finding 
%$\Delta$-significant levels can be thought of as a process of pruning the cluster
%tree estimators.

Below, we show that there is a one to one correspondence between $\Delta$-significant split levels  of the modified DBSCAN cluster tree estimator from \Cref{algorithm:MDBSCAN} and the  split levels of the population density under a slightly stronger covering condition than condition {\bf C} given above. Specifically, we assume the following.
\\
\\
{\bf C'.}
There exists a constant $\delta_0>0$ such that,
for any split level  $\lambda^* $ of $p$ and any $\delta \in \mathbb{R}$ with $|\delta| \le \delta_0$,
$ \{p\ge \lambda^*+\delta\} $ satisfies conditions {\bf C1} and {\bf C2} with
 constants $r_I$,$c_I$ and $C_I$ only depending on $p$. 
\\
\\
The only difference between {\bf C} and {\bf C'} is that while condition {\bf C}
assumes some regularity of $p$ only above split levels, {\bf C'}
requires the same type of regularity {\it around} split levels.

\begin{proposition} \label{prop:find split level}
	Suppose condition {\bf C'} and {\bf S($\alpha$)}  hold. 
	Let $\Delta =  2a_n +(4h/c_S)^{\alpha}  $
	where $a_n$ is  defined in  \eqref{eq:KDE risk} and 	$h=C_1n^{-1/(2\alpha+d) }$. 
	Suppose $p$ has finitely many split levels.
	Then,  with probability at least $1-1/n -O( h^{-d} \exp (-c  n^{\alpha/(2\alpha+d)})) $, 
	the following additional results hold:
	\\
{\bf 1.} 
		Let $\lambda^*$ be  a split level of the density $p$. 
		Suppose $ \mathcal C$ and $\mathcal C'$ are two open sets splitting  at $\lambda^*$ (see \Cref{definition:split at a level}) and that 
		\begin{equation}\label{eq:split regular}
		\min\{P\left(\C \cap \{p\ge\lambda^*+2\Delta\}\right),P\left(\C' \cap \{p\ge\lambda^*+2\Delta\}\right)\}>0.
		\end{equation}
		Then, there exists 
		a  $\Delta$-significant split level $\hatsplit$ of the cluster tree  estimator returned by the modified DBSCAN such that
		\begin{equation}
		\label{eq:split-consistent}
		|\lambda^*-\hatsplit | \le \Delta.
		\end{equation}
{\bf 2.} 
		Conversely, suppose that $\hatsplit$  is a $\Delta$-significant split level of the
		cluster tree estimator.
		Then there exists a split level $\lambda^*$ of $p$ such that
		\begin{equation}
		|\lambda^*-\hatsplit | \le \Delta.
		\end{equation}
\end{proposition}
Proposition \ref{prop:find split level}  says that, with high probability,
every $\Delta$-split level corresponds to a density split level
and that conversely, any split level of the density $p$ can be found if we have enough data.
To prune the cluster tree returned by the modified DBSCAN algorithm, it suffices to remove all the split levels that are not $\Delta$ significant.

\section{Densities with Gaps}

\label{section:gap}

We now consider the particular scenario where the  density $p$
exhibiting a jump
discontinuity in such a way that,  for all levels $\lambda$ in a given interval of
length $\epsilon$, the upper level sets $\{ x \colon p(x) \ge \lambda\}$
 do not change. The value of $\epsilon$
is referred to as  the {\it gap size.}    We provide a formal definition next, which we formulate within the general measure-theoretic language of \cite{steinwart:15} since the underlying densities are not continuous.  We recall that $\mathcal{L}$ denotes the Lebesgue measure on $\mathbb{R}^d$.

\begin{definition}[Distribution with a gap]\label{def:gap}
Let $P$ a probability measure on $\mathbb{R}^d$, absolutely continuous with respect to the Lebesgue measure. For any $\lambda^* > 0$, let $S_{\lambda^*}$ be the support of the sub-probability measure 
\[
A \mapsto  P(A \cap \{ x \colon p'(x) \geq \lambda^*\}), \quad A \text{ Lebesgue measurable,}
\]
where $p'$ is any density of $P$. 
Then, $P$ is said to have a gap at   $\lambda_*$ of size $\epsilon$, where $0 < \epsilon < \lambda^*$, if   $P(S_{\lambda^*}) > 0 $ with $\mathcal{L}(\partial S_{\lambda^*}) = 0$ and
\begin{equation}\label{eq:gap.steinwart}
S_{\lambda_* - \eta}  \backslash S_{\lambda_* } = \emptyset, \quad \forall \eta \in (0,\epsilon).
\end{equation}

%We will further assume that $\mathcal{L}(\partial S) =0$.
\end{definition}
We impose the condition that   $\mathcal{L}(\partial S_{\lambda^*}) = 0$ in order to avoid pathological cases.

The above definition is independent of the choice of the density of $P$. At the same time, it also implies that $P$ admits a Lebesgue density $p$ such that 
\begin{equation}\label{eq:Sl}
S_{\lambda^*} = \mathrm{cl} \left(  \{  x \colon p(x) \geq \lambda^* \}\right) \quad \text{and} \quad
\mathrm{cl}(S_{\lambda^*}^c) = \mathrm{cl} \left(  \{  x \colon p(x) \leq \lambda^*  - \epsilon \}\right).
 \end{equation}
Here, given any set $A$, $\mathrm{cl}(A)$ denotes the closure of $A$.  Indeed, it is not hard to see that, if $p'$ is any Lebesgue density of $P$, then the function 
\[
p(x)= \left\{ 
\begin{array}{ll}
\max\{ p'(x), \lambda^*\} & x \in S_{\lambda^*}\\
\min\{ p'(x), \lambda^* - \epsilon \} & x \in S_{\lambda^*}^c
\end{array}
\right.
\]
is also a density of $P$, and satisfies \eqref{eq:Sl}. 
Thus, with a slight abuse of notation, we may also speak of  ``the'' density $p$ even in this case, with the understanding that we are referring to any density of $P$ for which \eqref{eq:Sl} holds. See \Cref{fig:gapexample} for an illustration.

\begin{figure}[H]
	\begin{center}
		\includegraphics[scale=0.39]{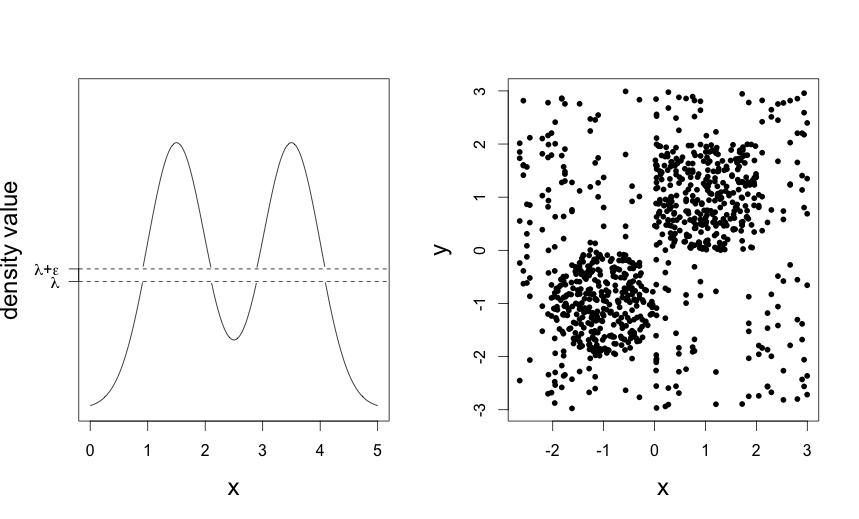}
		\caption{The left plot depicts a one dimensional density with gap of size $\epsilon$ at level $\lambda$.
			It is clear that $\{ \lambda < p \le \lambda +\epsilon \}$ is an empty set. The right plot depicts  500 
			i.i.d sampling from  a two dimensional density with a gap. It is clear that the density is low in the background and high in the  disk centered at (-1,-1) and  in the 
			square centered at (1,1). Finding the samples points with low density values can be thought of as outliers detection in this case. }
		\label{fig:gapexample}
	\end{center}
\end{figure}
%Suppose $P$ has a gap of size $\epsilon$ at $\lambda_*$ and let $p$ be any density of $P$. 
%Then, since
% $ \mathcal L( \{x :  \lambda_*<p(x) <\lambda_* +\epsilon\})=0$, we may set $p(x) = 0$ for all $x \in $ 
%Denote 
%$S := S^{\lambda_*}$. We can also characterize $S$ as 
%\begin{align}\label{eq:support gap}
%S  = \{ p > \lambda_* \} .
%\end%{align}
%This is because for any $\lambda \in (\lambda_*, \lambda_*+ \epsilon )  $ we have
%$  \{ p > \lambda  \} =S \cup \{x :  \lambda_*<p(x)  < \lambda\} .$
%Therefore the level set $\{ p > \lambda  \}$ is 
%unchanged (up to a measure $0$ set) for any $\lambda \in (\lambda_*, \lambda_*+ \epsilon )  $.
Though fairly restrictive, the scenario of a distribution with a gap is quite interesting for the purpose
of both clustering and level set estimation. 
Indeed,  this  situation encompasses the ideal clustering
scenario, depicted as  examples  in Figures
1 and 5 in the original DBSCAN paper \cite{ester1996density}, of a piece-wise
constant density that is low everywhere on its support with the exception of a few
connected, full-dimensional
regions, or clusters, where it is higher by a certain amount (in our case
the gap $\epsilon$). 
The
size of the gap parameter $\epsilon$ and the minimal distance among
clusters both affect the  difficulty of the clustering task, which becomes
harder as the parameters get smaller. 

To formally capture the dependence on the distance between clusters we set 
\[
S_{\lambda^*}  =\bigcup_{i=1}^I \mathcal C_i,
\] 
 where $(\mathcal C_1,\ldots, \mathcal  C_I)$ are disjoint, connected, (necessarily) closed sets of positive $P$ measure and let 
\begin{equation}\label{eq:sigma}
	\sigma = \min_{i\not = j}\mathrm{dist}(\mathcal
	C_i,\mathcal C_j) > 0
	\end{equation}
	be the minimal distance between them. The separation parameter $\sigma$ captures an aspect of the intrinsic difficulty of the clustering task that is complementary to the one quantified by gap parameter $\epsilon$: clusters that are at a small distance $\sigma$ from each other are hard to separate, for any given value of $\epsilon$. In our analysis, we let both the gap parameter $\epsilon$ and the separation parameter $\sigma$ vary with $n$ (though we do not make this dependence in out notation for ease of readability), thus allowing for harder clustering problems as a function of the sample size.

%\begin{remark}[Clustering consistency at the Gap Level]
%\label{subsection:Clustering at the gap level}
%Below we will study how DBSCAN cluster the sample points    above
%the gap level. We note that the clustering consistency using linkage type algorithms has 
%been well studied. See \cite{cuevas2001cluster}
%and more recent results in \cite{chaudhuri2014consistent}. Since the existing results do not focus on the gap setting as we do,
%we will restate  these existing results under the gap settings but will not claim the novelty of these results. 
%\end{remark}

%the clusters to be sufficiently well separated, as formalized in our next
%assumption:
%\\
%\\
%{\bf S.} (Separation condition).   Suppose  that  $ S  =\bigcup_{i=1}^I \mathcal C_i $ 
%where $(\mathcal C_1,\ldots, \mathcal  C_I)$ are disjoint, connected sets\footnote{ To
%ensure the connected components of $S$ are well defined,  we can assume    1)
%   the underlying density $p$ is continuous   on $\mathbb R^d \backslash \partial S$ 
%   and {\bf S}. This is because because the density $p$ is uniquely defined except on $\partial S$ 
%   and changing the density values at $\partial S$ do not change the pairwise distance between distinct connected components of $S$.}
% in $\mathbb{R}^d$ such that
%there exists a constant $\sigma $ such that
%	\[
%	\min_{i\not = j}\mathrm{dist}(\mathcal
%	C_i,\mathcal C_j)=\sigma >0.
%	\]
In the next simple result, 
we show that 
a flat version of the vanilla DBSCAN algorithm given in \Cref{algorithm:DBSCAN}, with suitable choices for the input parameters, can optimally estimate the clusters at $\lambda^*$ at a rate that depend explicitly on both $\epsilon$ and $\sigma$.

\begin{proposition}\label{prop:consistency-gap}
	Let $\{X_1,\ldots,X_n\}$ be an i.i.d. sample from a probability distribution $P$ that has a gap of size $\epsilon$ at level $\lambda^*$. Set $a_n=C_1 \sqrt{\frac{\log(n) +\log (1/ h)}{nh^d}}$  as in \eqref{eq:an} and suppose the
	input parameters $h$ and $k$ of the DBSCAN algorithm  are such that
	\begin{equation}\label{eq:gap-constrain-clustering}
	\sigma/4\ge h \ge C\left(\frac{\log(n)}{n \epsilon^2 }\right)^{1/d} \quad
	\text{and} \quad  k = \lceil n h^d
	V_d\lambda \rceil ,
	\end{equation}
for any $C>0$ such that $2 a_n < \epsilon$ and	any $\lambda$ in  $(\lambda^* -\epsilon + a_n,\lambda^*- a_n ]$.
	Then   with probability at least $1-1/n$:
	\begin{itemize}
		\item[i.] simultaneously over all  connected sets $A$ such that $ A_{2h}
		\subset \mathcal  C_i$, for some $i$, all the sample points in  $A$, if any, belong to the
		same connected component of $\mathbb{G}_{k,h}$;
		\item[ii.] simultaneously over all connected sets $A$  and $A'$ such that $ A_{2h}
		\subset \mathcal  C_i$ and $ A'_{2h}
		\subset \mathcal  C_j$, for some $i \not = j$, the sample points in  $A$ and
		$A'$, if any, belong to distinct connected components of $\mathbb{G}_{k,h}$. 
	\end{itemize}
\end{proposition}
The definition of $A_{2h}$ and $A_{-2h}$ can be found in \eqref{eq:in and out}.  Notice that the above results hold true for all $n$ large enough such that $a_n < \epsilon/2$.
 \Cref{prop:consistency-gap} implies the DBSCAN
algorithm will yield clustering consistency, in the sense of \cite{chaudhuri2014consistent}, provided that its input parameters fulfill \eqref{eq:gap-constrain-clustering} holds.
In particular, this result requires that the sample size relates to the gap parameter $\epsilon$ and the separation parameter $\sigma$ according to the inequality
\[
n \ge C  \frac{1}{\epsilon^2 \sigma^d},  
\]
for some constant $C$, depending on $d$. In fact, such scaling is nearly minimax optimal:   no other 
clustering algorithms can guarantee cluster consistency under the same
assumptions and with a better sample complexity as a function of both
$\epsilon$ and $\sigma$. This results follows from the lower bound guarantee given in Theorem
VI.1 of
\cite{chaudhuri2014consistent}, where we take notice that the parameters $\sigma$ and $\epsilon$ have different, though related, meaning; see  \Cref{section:lower bound clustering at the gap}.
In fact, the results in \Cref{prop:consistency-gap} can be further extended to hold over more general settings of arbitrary densities; see  \Cref{sec:general}.

We conclude by noting that the gap size $\epsilon$ and the separation parameters $\sigma$ quantify two very separate notions of intrinsic difficulty of clustering that are unrelated to each other, and the clustering problem becomes impossible whenever either one of them becomes so small to violate the lower bound \eqref{eq:clusteer.minimax}, regardless of the other. In particular, it is easy to give examples in which the clustering task is impossible to solve because $\epsilon$ is too small even if $\sigma$ is large, and the other way around. As a result, the overall hardness of the clustering problem around the gap is a combination of these two parameters. This is in contrast with the settings considered earlier, where, due to the smoothness of the underlying density, only one parameter is sufficient to capture separation among clusters.

\subsection{The Devroye-Wise estimator  of the Level Sets}
\label{subsection:level set at the gap}
The assumption of a density with gap allows us to carry out a further analysis
of the DBSCAN algorithm, showing that it is also minimax  optimal for
estimating the level set itself $S_{\lambda^*}$. For this purpose, the
DBSCAN algorithm reduces to the renown Devroye-Wise estimator:  see \cite{DW}.
Below we provide a novel, sharper analysis of  this estimator, where  we  allow the
size of the gap $\epsilon$ to decrease with $n$, and demonstrate
that its rate-optimal (again, we will not explicitly express this dependence in our notation for simplicity).  
To the best of
our knowledge, such scaling has not been previously established. 

Recall that the DBSCAN algorithm with inputs $k$ and $h$
outputs a set of nodes $\mathbb G_{h,k}$. One then may
construct the estimator
\begin{equation}
	\widehat S_h = \bigcup_{X_j\in \mathbb G_{h,k} }B(X_j,h)
\end{equation}
comprised of a union of balls around such points, and use it as an estimator of the corresponding high density region 
$
S = S_{\lambda^*} = \cup_{i=1}^I \mathcal{C}_i
$
consisting of all the clusters.

We measure the performance of any estimator $\widehat{S}$ with the Lebesgue
measure of its symmetric difference with $S$:
\[
\mathcal{L}\left( S \Delta \widehat{S} \right) =\mathcal L\left( S\cap \widehat S^c  \right) +\mathcal L\left( S^c \cap \widehat S \right).
\]
%In order to determine the difficulty
%of this estimation problem we will  need to use the gap size $\epsilon$ as a parameter which may depend on $n$. 
We will in addition impose the following condition:

	{\bf R.} (Level set regularity). There exists constants $h_0 >0$ and
	$C_0>0$ such that, for all $h \in (0,h_0)$,
	\[
	\mathcal L\left( S_h \setminus S_{-h}\right)\le C_0h,
	\]
	where $S_{h}$ and $S_{-h}$ are defined in \eqref{eq:in and out}.

	Condition {\bf R} is very mild. Indeed, if $\partial S$ is  $C^2$, then the set $N_h := S_h \backslash S_{-h} $  is the tubular neighborhood \citep[see, e.g.][]{hirsch2012differential}
	  of $\partial S$ in $\mathbb R^d$.
	In particular, every compact domain 
	in $\mathbb R^d$ with $C^2$ boundary satisfies condition {\bf R}.  In this case
	$C_0\preceq V_{d-1} |\partial S |$, where $| \partial S|$ denotes the surface volume of $S$.  We also note that {\bf R} is equivalent to the ``smooth boundary'' condition in \cite{steinwart:15}.

\begin{proposition}\label{prop:support estimate upper}
Let $\{X_1,\ldots,X_n\}$ be an i.i.d. sample from a probability distribution $P$ that has a gap of size $\epsilon$ at level $\lambda^*$. Let $a_n=C \sqrt{ \frac{\log(n) +\log (1/ h)}{nh^d}}$  be defined as  in \eqref{eq:an}.
	Suppose the input parameters  $(h,k)$ of the DBSCAN algorithm satisfy
	\begin{equation}\label{eq:gap-constrain}
		h_0\ge h \ge C\left(\frac{\log(n) }{n\epsilon^2}\right)^{1/d} \quad
	\text{and} \quad k = \lceil n h^d,
	V_d\lambda \rceil
	\end{equation}
	for any $C>0$ such that $2 a_n < \epsilon$ and	any $\lambda$ in  $(\lambda^* -\epsilon + a_n,\lambda^*- a_n ]$.
	Then  with probability at least $1-1/n$,
	$$\mathcal
	L(S \triangle \widehat S_h) \le 2C_0h ,$$ 
	where   $C_0$ is defined in  {\bf R} and  
	\begin{equation}\label{eq:Shat}
		\widehat S_h = \bigcup_{X_j\in \mathbb G_{h,k} }B(X_j,h)
	\end{equation}
\end{proposition}

Similarly to \Cref{prop:consistency-gap},  the above results hold true for all $n$ large enough such that $a_n < \epsilon/2$.
	If $P(X\in S) =1 $, the level set estimator $\widehat S_h$ is also a support estimator
	and 
	$\epsilon= \inf_{x\in S}p(x). $
	In this case, \Cref{prop:support estimate upper} says that if the lower bound on the  density
	vanishes no faster  than $O(n^{-1/2}) $, then support estimation is still possible.

Below we show that the error bound given in \Cref{prop:support estimate upper} is minimax optimal up to log factors.  Consider
 $\mathcal P^n(h_0,\epsilon) $, the class of probability distribution of $n$ i.i.d. random vectors in $\mathbb R^d$ whose common
density exhibit  a gap of  size $\epsilon$ such that condition \eqref{eq:Sl} holds, and
satisfying condition {\bf R} with parameter $h_0 > 0$. 
Then \Cref{prop:support estimate upper} shows that, for all $n$ large enough,
$$ \sup_{P\in \mathcal P^n(h_0,\epsilon)}
\mathbb{E}_P\left (S\triangle \widehat S_h\right) = O\left(\left(\frac{\log(n)}{n\epsilon^2}\right)^{1/d} \right),$$
provided that $h$ is of the order $
\left(\frac{\log(n)}{n\epsilon^2}\right)^{1/d}$. 
Our next result provides a nearly-matching  lower bound.

%\subsection{Optimal  estimation at the gap}\label{sec:lower-bound-gap}
%\dd{We need to modify this a bit,  since the difference between clustering and support estimation are delicate.}
%The results presented so far provide finite sample guarantees that hold for 
%fixed choices of the gap parameter $\epsilon$ and the separation parameter $\sigma$.
%In this section, we consider instead the asymptotic scenario in which the sample size grows
%unbounded and the parameters $\sigma$ and $\epsilon$ are also allowed to change
%with $n$, so as to make the tasks of estimating $S$ and of clustering the sample
%points at levels $\lambda \in (\lambda_*,\lambda^* )$ increasing difficult.
%Below we will
%show that the DBSCAN algorithm, with an appropriate choice
%of $h = h_n$, yields a minimax optimal estimator
%of $S$. Furthermore, the same bandwidth selection rule guarantees a minimax (up
%to a logarithmic factor) optimal scaling for cluster
%for the separation parameter $\sigma$, but not for the gap size parameter, whose
%minimax scaling can only be matched (up to a log factor) with  different choice
%of $h_n$.
%

\begin{proposition}
	\label{prop:lower gap level set}
	 There exist
	constants $h_0$ and $c$, depending only on $d$   such that   for any $\epsilon\le 1/4$, for all $n$ large enough here exist probability distributions $\{P_1,\ldots,P_M\}$ in $\mathcal P^n(h_0,\epsilon) $ such that
	$$\inf_{\widehat S} \sup_{i=1,\ldots,M} \mathbb{E}_{P_i}\left(\mathcal L(\widehat S \triangle S )
	\right)\ge c \min\left\{\left(\frac{1}{n\epsilon^2}\right)^{1/d}, 1\right\} ,$$
	where the infimum is with respect to all estimators of $S$. 

\end{proposition}

Thus, if $\frac{\log (n)}{n\epsilon^2} \to 0 $ as $n\to \infty$ (so that  condition \eqref{eq:gap-constrain} is eventually satisfied),  then the bounds given in 
\Cref{prop:support estimate upper}  and \Cref{prop:lower gap level set} match, up to a  $\log(n)$ factor. That is, with suitable choice of input, DBSCAN can optimally estimate the level set $S$ at the gap.

%\begin{remark}\label{remark:gap size vanish}
	The performance of the Devroye-Wise estimator is a well-established topic in the literature: see, e.g., Theorem 4 and 5 in  \cite{cuevas2004boundary}. Our contribution in this regard is two fold: we allow for an explicit dependence on the gap size parameter $\epsilon$  and deliver minimax lower bounds. Our rate of convergence
	confirms the intuition that a smaller gap size leads to a harder estimation problem. 
	%\end{remark}

\section{Discussion}
\label{section:discussion}

In this article we propose a new notion of consistency for estimating the clustering structure
under various conditions.  Our analysis shows that   the DBSCAN algorithm
  is minimax optimal.  Interestingly, the rates match, up to
log terms, 
minimax rates for density estimation in the supreme norm for H\"{o}loder smooth
densities. In particular, our
results provide a complete, rigorous justification to the plausible belief,
commonly held in
density-based clustering, that clustering is as difficult as density estimation.
In the rest of the discussion section, we will compare  our notion of separation  with other existing  ones in the literature. For the sake of exposition, we will follow the convention used in much of the literature on density-based clustering of assuming that the cluster tree of the data generating distribution in fact corresponds to the hierarchy of the upper level sets of a canonical density $p$. As explained in \cite{steinwart:15}, this definition is in general not well-posed, since different densities will yield different trees. 

\subsection{Hartigan consistency in \cite{hartigan1981consistency}}\label{section:hartigan}
We follow \cite{chaudhuri2010rates} and \cite{eldridge2015beyond}  in  defining Hartigan consistency in terms of the density cluster tree.
\begin{definition}[Hartigan consistency]
 Let $\widehat T_n$ be a cluster tree estimator constructed from i.i.d. data $\{X_i\}_{i=1}^n$ from a disribution $P$ with Lebesgue density $p$. For any pair of subsets 
 $A$ and  $A'$, let denote $A_n$ and  $A_n'$  be the smallest clusters of $\widehat T_n$ containing $A\cap  \{X_i\}_{i=1}^n$ and $A'\cap  \{X_i\}_{i=1}^n$, respectively. The cluster tree estimator $\widehat T_n$ is Hartigan consistent if, for any pair of sets
$A $ and $A'$ belonging to distinct connected components of $\{ x: p(x) \ge \lambda\} $ for some $\lambda$,
$  P (A_n \cap A_n' = \emptyset ) \to 1 \text{ as } n \to \infty$.
\end{definition}
It is immediate from \Cref{defi:delta-consistent} that a $\delta$-consistent cluster tree is also Hartigan consistent. While Hartigan consistency is a simple form of {\it point-wise} cluster tree consistency, which holds for each fixed pairs of disjoint clusters, $\delta$-consistency is a stronger guarantee, as it yields {\it uniform} consistency over all $\delta$-separated clusters and, furthermore, gives consistency rates depending on the value of the separation parameter $\delta$.

\subsection{Comparison with the Merge distortion metric }
The notion of $\delta$-separation is closely related to the notion of {\it merge distance}  introduced by \cite{eldridge2015beyond}, which we present next. 
\begin{definition}\label{defi:merge height}
Let $p$ and $q$ be Lebesgue densities in $\mathbb{R}^d$ and let $T_p$ and $T_q$ be the corresponding cluster density trees. The merge distortion distance between $T_p$ and $T_q$
is defined as
	\[
d_M (T_p, T_q) = \sup_{x,y \in \mathbb{R}^d} | m_p(x,y) -m_q(x,y)|,
	\]
	where, for a Lebesgue density $p$, 
	$$ m_p (x, y)=\sup\{ \lambda > 0 \in \mathbb R: \text{there exists $ C \in T_p(\lambda)$ such that } \{x,y\} \subset C \}.$$
%	Let $p$ be a Lebesgue density and $T_p(\lambda)$ denote the cluster tree generated by $p$ at level $\lambda$.
%	For any two clusters $A, A' \in T_p $, their merge height $m_p (A, A')$
%	is defined as 
%	$$ m_p (A, A')=\sup\{ \lambda > 0 \in \mathbb R: \text{there exists $ C \in T_p(\lambda)$ such that } A,A' \subset C .\}$$
%	If $T_p$ and $T_q$ are cluster density trees associated to the Lebesgue densities $p$ and $q$, respectively,  their merge distance is defined as
%	\[
%d_M (T_p, T_q) = \sup_{x,y \in \mathcal R^d} | m_p(x,y) -m_q(x,y)|,
%	\]
	\end{definition}
The original definition of merge distortion metric is, in fact, more general but, when specialized to our settings,  reduces to the one given above.

The merge distortion distance is closely related to the $L_\infty$ distance between densities. In fact, by Theorem 17 in \cite{eldridge2015beyond}, $d_M (T_p, T_q) \leq \| p - q\|_\infty $, so that, if $\{ p_n \}_n$ is a sequence of Lebesgue densities, then $\|p_n - p\| \rightarrow 0$ implies that  
$d_M (T_{p_n}, T_p) \rightarrow 0$. In fact, Lemma 1 in Appendix F of \cite{inference.tree} shows that, if $p$ and $q$ are continuous, then $d_M (T_p, T_q) = \| p - q\|_\infty$. As a result, for the class of continuous (and, in particular, H\"{o}lder smooth) densities, cluster consistency in the merge distortion distance  is equivalent to cluster consistency based on the $\delta$-separation criterion, which in turn is equivalent to estimation consistency of the underlying density in the $L_\infty$ norm.

The above statement immediately applies to the n\"{a}ive cluster tree estimator $T_{\widehat{p}_{h_n}}$ built using the level sets of any density estimator $\widehat{p}_{h_n}$ that is continuous and, as $h_n \rightarrow \infty$, consistent in the $L_\infty$ norm. Such estimator is of course computationally unfeasible even in low dimension. In fact, the DBSCAN-based procedures described above in \Cref{algorithm:DBSCAN,algorithm:MDBSCAN}, which are applicable in high-dimensional settings, are also consistent in the merge-distortion metric. To see this, and following the arguments of \cite{eldridge2015beyond}, it is sufficient to demonstrate the properties of minimality and separation, as defined in that reference, for the DBSCAN cluster-tree estimators. The separation property, which prevents the emergence of false clusters or over-segmentation, follows directly from the definition of $\delta$-separation; see also the pruning results of \Cref{section:consistent split level}. On the other hand, minimality avoids the occurrence of improper nesting and holds for the DBSCAN procedures we consider here in virtue of \Cref{lem:obvious} and \Cref{lemma:standard-level-inequality}. The fact that our algorithms produce cluster tree estimators that are consistent in the merge distance should not be  surprising, since $\delta$-consistency is  directly tied to consistency for density estimation in in $L_\infty$. As shown in \cite{eldridge2015beyond}, the robust single-linkage clustering algorithm of \cite{chaudhuri2014consistent} is also   consistent in the merge distance.

\subsection{Comparison with the $(\epsilon,\sigma)$-separation criterion of
	\cite{chaudhuri2014consistent} }
\label{subsection:Comparison with chaudhuri}
The criterion of $\delta$-separation we introduce in this paper is most useful when studying smooth densities. Nonetheless, it will be helpful to compare it to the
notion of $(\epsilon,\sigma)$-separation defined in
\cite{chaudhuri2014consistent}, which is applicable to arbitrary densities.
\begin{definition}[$(\epsilon,\sigma)$-separation criterion in  \cite{chaudhuri2014consistent}]\label{defi:Dasgupta} \
	\begin{itemize}
		\item[1.]Let f be a density supported on X $\subset \mathbb R^d$. We
		say that $A ,A' \subset X$ are $(\epsilon,\sigma )$-separated if there exists $S \subset X$
		(the separator set) such that (i) any path in $X$ from $A$ to $A'$ intersects
		$S$, and (ii) $\sup_{x\in S_{\sigma}} f(x) < (1-\epsilon) \inf_{x\in A_\sigma \cup A'_{\sigma}} f(x)$.
		\item[2.] Suppose an i.i.d samples $\{X_i\}_{i=1}^n$ is given. An estimate of  the cluster tree is said to be $(\epsilon,\sigma)$ consistent if for any pair $A$ and $A'$ being $(\epsilon,\sigma)$  separated, the smallest cluster containing $A\cap \{X_i\}_{i=1}^n$ is disjoint from the smallest cluster containing $A'\cap \{X_i\}_{i=1}^n$. 
	\end{itemize}
\end{definition}
In the following result we make a straightforward connection between the $\delta$-separation and $(\epsilon,\sigma)$-separation.
\begin{lemma} 
	\label{lemma:defi-equivalence} Assume that $p \in \Sigma(L, \alpha)$ with $\alpha\le 1$  and that
	$A$ and $A'$ are $\delta$-separated. Then, $A$ and $A'$ are $(\epsilon,\sigma)$-separated   with 
	\begin{equation}
	\label{eq:relations}
	S= \{x \colon p(x)
	\le \lambda -\delta\}, \quad \epsilon =\delta/(3\lambda) \quad \text{and} \quad \sigma^\alpha=\delta/(3L),
	\end{equation}
	where $\lambda = \inf_{z\in A\cup A'} p(z)$.
\end{lemma}

\begin{proof}[Proof of lemma \ref{lemma:defi-equivalence}]
	Denote $\lambda= \inf_{z\in A\cup A'} p(z).$
	Suppose for the sake of contradiction that there is a path $l$ connects $A$ and $A'$ and that $l\cap \{p\le \lambda -\delta\} =\emptyset$.
	Then by the continuity of $p$ and the compactness of $l$, there exist $\gamma>0$ such that $l\subset \{p\ge \lambda-\delta+\gamma \} $.
	Thus $A$ and $A'$ belongs to the same path connected component of $\{p\ge \lambda-\delta+\gamma\}$. 
	Since 
	$\{p\ge \lambda-\delta+\gamma\} \subset \{p>\lambda-\delta\}$, $A$ and $A'$ belongs to the same path connected component of $\{p> \lambda\}$. Since $\{p>\lambda\}$ is an open set, $A$ and $A'$ be
	belongs to the same connected component of $\{p> \lambda\}$.
	This is a contradiction. 
	\\
	\\
	Let   $\sigma^\alpha=\delta/(3L)$ and $ \epsilon=\delta/3$, then for any $x\in S_{\sigma}$, $p(x)\le \lambda-\delta+L\sigma^\alpha =\lambda-2\delta/3$.
	Similarly if $x \in A_{\sigma}\cup A'_{\sigma} $, $p(x)\ge \lambda - L\sigma^\alpha =\lambda-\delta/3$.
	Thus 
	\begin{align*}
		(1-\epsilon) \inf_{x\in A_\sigma \cup A'_{\sigma}} f(x)
		> &(1-\epsilon)\lambda-\delta/3= \lambda-2\delta/3
		>\sup_{x\in S_{\sigma}} f(x)
	\end{align*}
\end{proof}

According to the separation criterion in \Cref{defi:Dasgupta},  two
clusters can be
$(\epsilon,\sigma)$-separated for many the values of
$\epsilon$ and $\sigma$. In particular,
by taking the separator set to be larger, it is easy to produce examples of
$\delta$-separated clusters that are also $(\epsilon,\sigma)$-separated  
such that  $\delta$ is big but $\sigma$ is small. This is simply because
$\sigma $ is heavily associated with $S$. And conversely, by taking an almost
flat density function, it is possible to have a very large $\sigma$ and very small $\delta$. 

We remark that when $\alpha> 1 $, there is no obvious relationship between the parameter $(\sigma,\epsilon)$ in
\Cref{defi:Dasgupta} and $\delta$  in
\Cref{defi:delta-consistent} as that in \Cref{lemma:defi-equivalence}. For $\alpha>1$, while $p \in
\Sigma(L,\alpha)$ implies that $p$ is  Lipschitz continuous, the 
Lipschitz constant in this case does not depend  on $L$ and $\alpha$ in a simple
manner. As a result, the parameter $\sigma$,  representing the distance between
connected components of upper level sets of $p$,
is not straightforwardly related to $\delta$.

	\bibliographystyle{plainnat}
	\bibliography{citation}
\appendix
	
%!TEX root = ./dbscan_JMLR.tex

\section{Topological Preliminaries}
\label{section:topology}
For completeness, we review the definition of connectedness from the general topology.
\begin{definition}[  \cite{munkres2000topology} Chapter 3]
Let $U$ be any nonempty subset in $\mathbb R^d$. 
Then $U$ is said to be connected, if,
for every pair of open subsets  $A ,A' $ of $ U$ such that 
$A\cup A'=U $,    we have either $A=\emptyset $ or $A'=\emptyset$. The maximal connected subsets of $U$ are called the connected 
components of $U$.
\end{definition}
We briefly  explain  why the connected components
naturally introduce a hierarchical structure to the level sets of $p$. Let $\lambda_1>\lambda_2$,
 so we have $\{p\ge \lambda_1\} \subset \{p\ge \lambda_2\}$. 
\begin{itemize}
\item Suppose 
$A$ is any subset of $\mathbb R^d$, and $A$ belongs to the same connected component of $\{p\ge \lambda_1\} $. 
Then
$A$ is contained in the same connected component of $\{p\ge \lambda_2\} $.
\item Suppose $A \cup A' \subset \{p\ge \lambda_1\}$ and they  belong to distinct connected components of $\{p\ge \lambda_2\} $
Then  $A$ and $A'$  are not contained in the same connected component of $\{p\ge \lambda_1\} $.
\end{itemize}

We also review a closed related concepts, which is call the path connectedness in general topology.
\begin{definition}
We say that a subset $U\subset \mathbb R^d$ is path connected if for any $x,y\in U$,
there exists a path continuous $\mathcal P:[0,1]\to U$ such that $\mathcal P(0)=x $ and $\mathcal P(1)=y$.
\end{definition}
 The main reason we introduce the path connectedness is that if $U$ is an open set in $\mathbb R^d$, then
 $U$ is connected if and only if it is path connected. Therefore a simple but useful consequence is that for any $\lambda$, 
 the connected components of $\{p>\lambda\}$ are also the path connected components. 
  \\
We will repeatedly use these topological properties  in  our analysis without further mentioning.
The proofs of them  are omitted and  can be found in  \cite{munkres2000topology}  or any other 
books on general topology.

\section{Proofs from Section  \ref{section:holder}}
\label{appendixa}

We begin by justifying  \eqref{eq:variance of density}.
Since this is a well known result, we  simply use a result of
\cite{sriperumbudur2012consistency}. We will assume the  following condition 
for the kernel $K$ which is fairly standard in the non-parametric literature.

\begin{itemize}
	\item[VC.] The kernel $K: \mathbb{R}^d \rightarrow \mathbb{R}$ has bounded
	support and integrates to 1. Let $\mathcal{F}$ be the class of functions
	of the form
	\[
	z \in \mathbb{R}^d \mapsto
	K\left(x-z \right), \quad z \in \mathbb{R}^d.
	\]
	Then, $\mathcal{F}$ is a uniformly bounded VC class: 
	there exist positive constants $A$ and $v$ such that
	$$\sup_{P}\mathcal N (\mathcal F,L^2(P),\epsilon \|F \|_{L^2(P)})\le (A/\epsilon)^{v} ,$$
	where $\mathcal N(T,d,\epsilon)$ denotes the $\epsilon$-covering number of the metric space $(T,d)$, F is the envelope function of $\mathcal F$ and the 
	sup is taken over the set of all probability measures on $\mathbb R^d$. The
	constants $A$ and $v$ are called the VC characteristics of the kernel.
\end{itemize}
The assumption VC holds for a large class of kernels, including any compact supported polynomial kernel and the Gaussian kernel. 
See \cite{nolan1987u} and \cite{gine2002rates}.

\begin{proposition}[\cite{sriperumbudur2012consistency}]
	\label{prop:Steinward}
	Let $P$ be the probability measure on $\mathbb R^d$ with Lebesgue density
	bounded by $\|p\|_{\infty}$ and
	assume that the kernel $K$ belongs to $L^\infty(\mathbb R^d )\cap L^2(\mathbb
	R^d)$ satisfies the  VC assumption. Then for any $\gamma > 0$ and $h>0$, there
	exists an absolute constant $C$ depending on the VC characteristic of $K$ such
	that, with probability no smaller than $ 1- e^{-\gamma}$,
	\begin{equation}\label{eq:VC kernel} 
	\|p_h-\widehat p_h\|_{\infty}  \le 
	\frac{C}{nh^d}\left(\gamma + v\log\frac{2A}{\sqrt{h^d\|p\|_{\infty}\|K\|^2_{2}}}\right)
	+C\sqrt{\frac{2 \|p\|_{\infty}}{nh^d}}\left(\gamma\|K\|^2_{\infty} +v\|K\|^2_2\log\frac{2A}{\sqrt{h^d\|p\|_{\infty}\|K\|^2_2}}\right)  \nonumber
	\end{equation}
\end{proposition}

\begin{lemma}\label{lemma:merge height}
Suppose $A'$ and $A'$ are two clusters of $T_p(\lambda) $ and they are $\delta $-separated. Then their merge height $m_f(A, A')$ satisfies
$$m_f(A, A') \le \lambda -\delta .$$

\end{lemma}
\begin{proof}
By \Cref{defi:path-delta-separated}, there $A$ and $A'$ belong distinct connected components of $\{ p> \lambda -\delta\} $.
For the sake of contradiction, suppose
$$m_f(A, A') >  \lambda -\delta.$$
Therefore there exists $\lambda'$ such that 
$m_f(A, A')  > \lambda' > \lambda -\delta $ and that   by \Cref{defi:merge height} 
$A $ and $A'$ belong the same connected component of $\{ p\ge \lambda'\}  $. This is a contradiction because 
$ \{ p\ge \lambda'\} \subset \{ p> \lambda -\delta\}$.
\end{proof}

\subsection{Proofs in \Cref{subsection:split level} }

\begin{proof}[Proof of proposition \ref{lemma:well-defined-2}]
	To show proposition \ref{lemma:well-defined-2}, we begin by introducing a standard topology lemma. 
	\begin{lemma}\label{lemma:well-defined}
		Suppose $p:\mathbb R^d \to \mathbb R$ are compactly supported. If $A$ and $A'$ are in the same connected components of $\{p\ge \lambda_i \}$ for $i=1,2,\ldots,\infty$ and that $\lambda_i\le \lambda_{i+1}$,then $A$ and $A'$ are in the same connected components of $\{p\ge \lambda_{0} \}$, where $\lambda_0=\sup_i\lambda_i$.
	\end{lemma}
	\begin{proof}[Proof of lemma \ref{lemma:well-defined}]
		Let $\mathcal C_i$ be the connected component of $\{p\ge \lambda_i \}$ that contains $A$ and $A'$. Thus $\mathcal C_i$ are compact and connected. Since $\mathcal C_{i+1}\subset \mathcal C_i$ for all $i\ge 1 $, $\bigcap_{i=1}^{\infty}\mathcal C_i$ is connected. Thus
		$A ,A' \subset \bigcap_{i}\mathcal C_i\subset \{p\ge \lambda_0 \} $.
	\end{proof}
	Consider $$\lambda^{*}=\sup\{\lambda: A , A' \ \text{belongs to the same connected components of }\{p\ge \lambda \}\}$$
	Then $\lambda_2 \le \lambda^{*} \le \lambda_1$.
	By lemma \ref{lemma:well-defined} , $A$ and $A'$ are in the same connected components of $\{p\ge \lambda^{*} \}$. Thus $\lambda_2\le\lambda^{*} <\lambda_1$.\\
	In order to show that $\lambda^*$ is split level, it suffices to show that $A$ and $A'$ are in the different connected components of 
	$\{p>\lambda^{*} \} $. Suppose for the sake of contradiction that $A$ and $A'$ are connected in $\{p>\lambda^*\}$. Then $A$ and $A'$ are path connected as $\{p>\lambda^*\}$ is open. Thus  there exist $\mathcal P$ connects $A$ and $A'$ in $\{p>\lambda^{*} \} $. Since $\mathcal P$ is compact, $p(\mathcal P)>\lambda^{*}$ implies that there exists $ a>0$ such that $\lambda^*+a<\lambda_1 $ and $\mathcal P , A, A'  \subset \{p\ge\lambda^{*}+a\}$.
	Thus $A$ and $A'$ belong to the same connected component of $\{p\ge \lambda^*+a \}$.
	This is a contradiction because by construction of $\lambda^*$, $A$ and $A'$ belongs to the different connected components of $\{p\ge \lambda^{*}+a \}$.
\end{proof}
\
\
\newline
\begin{proof}[Proof of corollary \ref{coro:delta split}]
	
	Suppose $A$ and $A'$ are $\delta$ separated with respect to $\lambda$. 
	Then $A$ and $A'$ belongs to distinct connected components of  $\{p>\lambda-\delta \}$ where $\lambda=\inf_{x\in A\cup A'} f(x) $.  
	Let $0<\epsilon\le \delta $ be given. Then since $\{p\ge \lambda-\delta+\epsilon\}\subset \{p> \lambda-\delta \} $, 
	$A$ and $A'$ belongs to distinct connected components of  $\{p\ge\lambda-\delta+\epsilon\}$.\\
	Since $\mathbb R^d=\{p\ge 0\}$ is connected, $A$  and $A'$ belongs to the same connected component of $\{p\ge 0\}$.
	By proposition \ref{lemma:well-defined-2}, 
	there exists $0\le \lambda^*<\lambda-\delta+\epsilon$ such that 
	$A$ and $A'$ in the same connected component of $\{ p\ge \lambda^*\}$ and in
	different connected components of $\{p>\lambda^*\}$.
	By taking $\epsilon\to 0$, the claimed result follows.
	
\end{proof}

\subsection{Proofs in sections \ref{section:upper bounds}}

 \begin{proof}[Proof of  \Cref{lemma:DBSCAN-level-inequality}]
 	From the proof of  \Cref{lemma:standard-level-inequality}, it can be see that 
 	$$\left\{ p\ge \lambda +C\frac{\log(n)}{n^{\alpha/(2\alpha+d)}} \right \}  \cap \{X_i\}_{i=1}^n \subset \widehat D(\lambda) \subset \left\{ p\ge \lambda -C\frac{\log(n)}{n^{\alpha/(2\alpha+d)}} \right \} .$$
 	Thus for any $y\in B(X_j,h)$ for some $X_j \in \hat D(\lambda)$, 
 	$$p(y) \ge p(X_j)-Lh^\alpha\ge \lambda- C\frac{\log(n)}{n^{\alpha/(2\alpha+d)}} -Lh^\alpha,$$
 	where the first inequality follows from $|p(y)-p(X_j) | \le Lh^\alpha$.  Therefore the above display implies 
 	$$\bigcup_{X_j\in \hat D(\lambda)}  B(X_j,h)\subset \left\{ p\ge \lambda -C\frac{\log(n)}{n^{\alpha/(2\alpha+d)}}-Lh^\alpha   \right \}.$$
 	For the other inclusion, let $x \in  \left\{ p\ge \lambda +C\frac{\log(n)}{n^{\alpha/(2\alpha+d)}}+Lh^\alpha \right \}$.
 	Then $\hat p_h(x) \ge \lambda +Lh^\alpha $. Thus $B(x,h)\cap \{X_i\}_{i=1}^n\not =\emptyset$, or else $\hat p_h(x)=0$.
 	Let $X_j \in B(x,h)$. Therefore 
 	$$p(X_j) \ge p(x) -Lh^\alpha \ge \lambda + \frac{\log(n)}{n^{\alpha/(2\alpha+d)}} .$$
 	Thus $\widehat p_h(X_j)\ge \lambda,$ which means that $X_j \in \widehat D(\lambda) $.
 	So $x\in  \bigcup_{X_j\in \hat D(\lambda)}  B(X_j,h) $ and the first inclusion follows.
 	
 \end{proof}

 \begin{proof}[Proof of \Cref{prop:MDBSCAN}] 
 	Let $\mathcal B$ be the event that 
 	$$\mathcal  B=  \{\sup_{x\in\mathbb R^d} | \widehat p_h(x) -p(x)| \le a_n\}$$
 	By \Cref{prop:Steinward}, we can choose  $a_n$ 
 	so that $P(\mathcal B) \ge 1-1/n$ and that 
 	$a_n=O\left( \left( \frac{\log(n)}{n}\right)^{\alpha/(d+2\alpha)} \right).$  All the argument will be made on the good event $\mathcal B$.\\
 	\\
 	Observe that $p$ has connected support. Therefore $\lambda =0$ is not a split level. Assume that 
 	$\lambda_0 = \min\{\lambda^*:  \lambda^* \text { is a split level of  } p \} $. Then $\lambda_0>0$.  
 	If $h=O(n^{-1/(2\alpha+d) })$,
 	for large $n$, we have $ 2a_n +(4h/c_S)^\alpha<\min\{\delta_S,\delta_0\}$. Take
 	$$\delta \ge 2a_n +(4h/c_S)^\alpha /c_S. $$
 	Let $A$ and $A'$ are  two sets being $\delta$-separated
 	and let $\lambda= \inf_{x\in A\cup A'} p(x)$. Since $A$ and $A'$
 	are in distinct connected components of $\{p> \lambda-\delta\} $,
 	by proposition \ref{lemma:well-defined-2}
 	there exists $\lambda^*$ being a split level of 
 	$p$ such that $\lambda^*\le \lambda-\delta$ and that 
 	$A$ and $A'$ belongs to distinct connected components of $\{p>\lambda^*\}. $
 	Thus $A$ and $A'$ belong to distinct connected components of $\{p>\lambda'\}, $
 	where 
 	$$
 	\lambda'=\lambda^* +2a_n+(4h/c_S)^\alpha /c_S. 
 	$$
 	Let $\{\mathcal C_{k}\}_{k=1}^K$ be the collection of connected components of $\{p>\lambda' \}$.
 	Thus we have $A\subset \mathcal C_k$ and $A'\subset \mathcal C_{k'}$ for some $k\not =k'$.
 	In order to show the smallest cluster containing $A\cap \{X_i\}_{i=1}^n$ and $A'\cap \{X_i\}_{i=1}^n$ are disjoint with high probability,
 	it suffices to show the following statement.
 		\\
 		\\ 
 		$\bullet$ Let $A$ and $A'$ be two connected subsets of $\{ p>\lambda'\}$ and belong to two distinct connected components of 
 		$\{p>\lambda^*\}$.
 		Then 
 		the smallest cluster containing
 		$A\cap\{X_i\}_{i=1}^n$ and $A'\cap \{X_i\}_{i=1}^n$
 		are disjoint with high probability. 
 	\\
 	\\
 	Note that this observation reduce the original statement which concerns with
 	generic $\delta$-separated sets to the current statement  which only concerns with one level near the split level. 
 	Since there are finitely many split levels, a simple union bound will suffice to show the $\delta$ consistency of the cluster tree returned by  \Cref{algorithm:MDBSCAN}.
 	\\ 
 	\\
 	The proof will be completed by the following two claims.
 	\\
 	\\
 	{\bf Claim 1.} If $A$ is a connected subset of $\{p>\lambda'\}$,
 	then $A\cap \{X_i\}_{i=1}^n$ is in the same connected component
 	of 
 	\begin{equation}
 	\widehat L(\lambda'-a_n) : = \bigcup_{\{ X_j :  \widehat D(\lambda'-a_n) \}} B(X_j, 2h).
 	\end{equation}
 	\begin{proof}
 		It suffices to show that 
 		\begin{equation}\label{eq:C MDBSCAN}
 		\{p>\lambda'\}\subset  \widehat L(\lambda'-a_n).
 		\end{equation}
 		Since for large $n$,
 		$$a_n +(4h/c_S)^\alpha\le \delta_0, $$
 		By {\bf C2} there exists $\mathcal N_{h}\subset \{ p>\lambda'\} $ with $card(\mathcal N_{h})\le A_c (h)^{-d}$ such 
 		that $\mathcal N_{h} $ is a $h$ cover.
 		\\
 		Since $\{p>\lambda'\}$ satisfies the inner cone condition {\bf C1},
 		$$P(B(x,h)\cap \{p>\lambda'\}) \ge\lambda^*c_IV_d h^d \ge \lambda_0c_IV_dh^d .$$
 		So there exists $c'_I$ only depending on $d$ and $c_I$ such that 
 		$$ P(\{ \{X_i\}_{i=1}^n\cap B(x,h)\cap \{p>\lambda'\} =\emptyset \}) \le (1-\lambda_0c_IV_dh^d)^n \le \exp (-c_I' \lambda_0 n^{2\alpha/(\alpha+d)}) =o(n^{-2}),$$
 		where the second inequality follows from $h=O(n^{1/(2\alpha+d)})$ and the equality follows from \\
 		$\lambda_0n^{2\alpha/(2\alpha+d)}/\log(n) \to \infty  $
 		and $n$ being large enough.
 		Consider the  event 
 		$$\mathcal A = \{ \{X_i\}_{i=1}^n\cap B(x,h)\cap \{p>\lambda'\} \not =\emptyset  \text{ for all $x\in \mathcal N_{h}$} \}.$$
 		By the union bound 
 		\begin{equation}\label{eq:MDBSCAN dependent}
 		P(\mathcal A^c) \le card (\mathcal N_{h})  \exp (-c_I' \lambda_0 n^{2\alpha/(2\alpha+d)}) = A_ch^{-d}  \exp (-c_I' \lambda_0 n^{2\alpha/(2\alpha +d)})=o(1).
 		\end{equation}
 		So for any $y\in\{p>\lambda'\}$ , there exists $x\in \mathcal N_{h}$ such that $| y-x| \le h$. Under event $\mathcal A$
 		there exists $X_j \in \{p >\lambda'\} $ such that $ |X_j-x|\le h.$ Therefore $y\in B(X_j,2h)$. Since 
 		$$ X_j \in \{X_i\}_{i=1}^n\cap \{p>\lambda'\}\subset \widehat D(\lambda'-a_n) ,$$ the claim follows.
 	\end{proof}
 	To finish the proof of the theorem, we still need to show  at level $\lambda'-a_n$
 	the data points $A\cap \{X_i\}_{i=1}^n$ and $A'\cap \{X_i\}_{i=1}^n$  are 
 	contained in distinct clusters.  Therefore the following claim finish the proof.\\
 	\\
 	{\bf Claim 2.} There exists a partition  $\{ S_i\}_{i=1}^I$ of $\widehat D (\lambda'-a_n)$ such that 
 	$A\cap \{X_i\}_{i=1}^n$ and $A'\cap\{X_i\}_{i=1}^n$ belong to distinct subsets of the partition and that 
 	data points in distinct subsets of the partition are mutually disconnected. 
 	\begin{proof}
 		Let $\{B_i\}_{i=1}^I $ be the collection of connected components 
 		of $\{p\ge  (4h/c_S)^{\alpha} +\lambda^* \}. $ 
 		Since $A$ and $A'$ belong to distinct connected components of $\{p>\lambda^*\}, $
 		and $\lambda^*< (4h/c_S)^{\alpha} +\lambda^*$,
 		$A$ and $A'$ are contained in distinct elements of $\{B_i\}_{i=1}^I $. 
 		From condition {\bf S} , 
 		\begin{equation}\label{eq:mutual distance}\min_{i\not = j} d(B_i,B_j) \ge 4h.
 		\end{equation}
 		Note that $\widehat D(\lambda'-a_n) \subset  \{p\ge (4h/c_S)^\alpha +\lambda^* \}$ as a consequence of event $\mathcal B$. 
 		Thus  $S_i=B_i \cap \widehat D(\lambda'-a_n)$ form a partition of $\widehat D(\lambda'-a_n)$.
 		Let 
 		$$L_i = \bigcup_{X_j \in S_i}  B(X_j, 2h).$$
 		By \eqref{eq:mutual distance},  $L_i\cap L_j=\emptyset $ if $i\not = j$.  This shows that data points in distinct subsets of the partition $\{S_i\}_{i=1}^I$ are mutually disconnected at the graph $\mathbb C(h,\lambda'-a_n)$.
 	\end{proof}
 \end{proof}

 \begin{proof}[Proof of  \Cref{lemma:properties of morse}]
 	\
 	\\
 	{\bf Step 1.} In this step we show that condition {\bf S(2)} holds. 
 	Consider an arbitrary split level $\lambda$, and two connected components $C_1$, $C_2$. If 
 	\begin{equation*}
 	\inf_{\delta>0} d(C_1\cap \{p\ge \lambda+\delta\}, C_2\cap \{p\ge \lambda+\delta\})>0
 	\end{equation*}
 	then we have $d(C_1\cap \{p\ge \lambda\}, C_2\cap \{p\ge \lambda\})>0$, and the thesis is trivial. Thus
 	assume that
 	\begin{equation*}
 	\inf_{\delta>0} d(C_1\cap \{p\ge \lambda+\delta\}, C_2\cap \{p\ge \lambda+\delta\})=0,
 	\end{equation*}
 	i.e. 
 	\begin{equation*}
 	\lim_{\delta\to0} d(C_1\cap \{p\ge \lambda+\delta\}, C_2\cap \{p\ge \lambda+\delta\})=0.
 	\end{equation*}
 	Thus there exists $y_0\in \{p=\lambda\}$, and points $y_{1,2}^\delta\in C_{1,2}\cap \{p\ge \lambda+\delta\}$ such that  
 	\begin{equation*}
 	y_{1,2}^\delta\overset{\delta\to 0}\to y_0,\qquad |y_1^\delta-y_2^\delta|=d(C_1\cap \{p\ge \lambda+\delta\}, C_2\cap \{p\ge \lambda+\delta\}).
 	\end{equation*}
 	It is straightforward to check that $p(y_1^\delta)=p(y_2^\delta)=\lambda+\delta$.
 	
 	The thesis is now rewritten as $|y_1^\delta-y_2^\delta|\ge c_S\delta^{1/2}$ for some constant $c_S>0$ and all
 	sufficiently small $\delta$. Since split levels are also critical, $\nabla p(y_0)=0$; since
 	$p$ is a Morse function, $\nabla^2 p(y_0)$ is non-degenerate. By Taylor formula we have
 	\begin{equation}
 	\delta=p(y^\delta_j)-p(y_0) = (y^\delta_j-y_0)^T \nabla^2 p(y_0)(y^\delta_j-y_0)/2 +O(|y^\delta_j-y_0|^3),\qquad j=1,2,
 	\label{dist}
 	\end{equation}
 	and, as $\nabla^2 p(y_0)$ is non-degenerate, it follows $|y^\delta_j-y_0|=O(\delta^{1/2})$, i.e.
 	there exist constants $c_1,c_2,\delta_0>0$ such that 
 	$$c_1\delta^{1/2}\le |y^\delta_j-y_0|\le c_2\delta^{1/2} \qquad \text{for all } \delta\in (0,\delta_0).$$
 	We can estimate $c_2$ from below: denoting by 
 	\begin{equation*}
 	a:= \max\{|e_1(y_0)|,|e_2(y_0)|\},\qquad e_1(y_0),e_2(y_0) = \text{eigenvalues of } \nabla^2p(y_0),
 	\end{equation*}
 	\eqref{dist} gives
 	\begin{equation*}
 	(y^\delta_j-y_0)^T \nabla^2 p(y_0)(y^\delta_j-y_0)\le ac^2_2|y^\delta_j-y_0|^2, 
 	\end{equation*}
 	hence $c_2\ge \sqrt{2/a}$.
 	By the Lipschitz regularity of the gradient, i.e. hypothesis
 	\begin{equation*}
 	|\nabla p(x)-\nabla p(y)|\le L|x-y|
 	\end{equation*}
 	for some $L>0$, we have
 	\begin{equation}
 	|\nabla p(y_1^\delta)-\nabla p(y_0)|=|\nabla p(y_1^\delta)|\le L|y^\delta_1-y_0|\le Lc_2\delta^{1/2}.
 	\label{gr}
 	\end{equation}
 	Consider now the segment $[y_1^\delta,y_2^\delta]$ between $y_1^\delta$ and $y_2^\delta$: since
 	$y_j^\delta\in C_j\cap \{p\ge \lambda+\delta\}$ ($j=1,2$), and $C_j\cap \{p\ge \lambda+\delta\}$
 	are disconnected for all $\delta>0$, there exists some point $z\in [y_1^\delta,y_2^\delta]$ such that $p(z)<\lambda+\delta/2$. 
 	By Taylor's formula we then have
 	\begin{equation*}
 	p(z)=p(y_1^\delta)+\nabla p(y_1^\delta) \cdot (z-y_1^\delta) + (z-y_1^\delta)^T\nabla^2 p(y_1^\delta) (z-y_1^\delta)/2
 	+O(|z-y_1^\delta|^3)
 	\end{equation*}
 	If inequality $|y_1^\delta-y_2^\delta|\le k\delta^{1/2}$ were to holds
 	for some $k>0$, then since the domain is compact and $\nabla^2 p\in C^2$,
 	denoting by 
 	\begin{equation*}
 	A:=\sup_x \Big(\max\{|e_1(x)|,|e_2(x)|\}\Big),\qquad e_1(x),e_2(x) = \text{eigenvalues of } \nabla^2p(x),
 	\end{equation*}
 	we have
 	\begin{align*}
 	|p(z)-p(y_1^\delta)| &\le |\nabla p(y_1^\delta)|\cdot |z-y_1^\delta| + |\nabla^2 p(y_1^\delta)|\cdot |z-y_1^\delta|^2/2\\
 	&\le|\nabla p(y_1^\delta)|\cdot |y_1^\delta-y_2^\delta| + |\nabla^2 p(y_1^\delta)|\cdot |y_1^\delta-y_2^\delta|^2/2
 	\overset{\eqref{gr}}\le (Lkc_2+k^2A/2)\delta.
 	\end{align*}
 	Since $p(y_1^\delta)=\lambda+\delta$, and $p(z)<\lambda+\delta/2$, we need $Lkc_2+k^2A/2 >1/2$, hence
 	\begin{equation*}
 	k\ge A^{-1}(\sqrt{L^2c_2^2+A}-Lc2),
 	\end{equation*}
 	i.e. 
 	\begin{align*}
 	|y_1^\delta-y_2^\delta|&=d(C_1\cap \{p\ge \lambda+\delta\}, C_2\cap \{p\ge \lambda+\delta\})\\
 	&\ge A^{-1}(\sqrt{L^2c_2^2+A}-Lc_2)\delta^{1/2} \ge A^{-1}(\sqrt{2L^2/a+A}-L\sqrt{2/a})\delta^{1/2}.
 	\end{align*}
 	\
 	\\
 	\\
 	{\bf Step 2.} In this step we show that condition {\bf C} holds. 
 	\\
 	\\
 	{\bf Proof of C1.} Since a Morse function has only isolated non degenerate critical points, and
 	an isolated set in a compact domain is also finite, we infer that $\nabla p(x)=0$ only
 	for finitely many $x$. In particular, since $\lambda^*$ are split levels,
 	and $\{p=\lambda^*\}$ contains a critical point, there exist sufficiently small $\delta_1$, $\delta_2>0$
 	such that $\{\lambda^*+\delta_1\le p\le \lambda^*+\delta_2\}$ contains no critical points (since there are only finitely
 	many critical points). Since
 	the level sets are orthogonal to the gradient, we infer that $\{p= \lambda^*+\delta_1\}$
 	is smooth. In particular, $\{p= \lambda^*+\delta_1\}$ it satisfies the inner cone property with $c_I=1/2$.
 	
 	\bigskip
 	
 	The key difficulty in extending the above argument to $\{p>\lambda^*\}$ (instead of just $\{p\ge \lambda^*+\delta_1\}$
 	with $\delta_1>0$) is that the norm of gradient $|\nabla p|$ can approach zero as $\delta_1\to 0$,
 	since $\{p=\lambda^*\}$ is a split level, hence it contains critical points.
 	
 	The Morse function requirement, however, gives the ``bare minimum'' regularity to ensure C1.
 	We aim to prove, by contradiction, that $\{p\ge\lambda^*\}$ also satisfies C1, i.e.
 	the boundary $\{p=\lambda^*\}$ does not exhibit cusps. If 
 	a cusp were to appear, then there exist arc-length parameterized curves $\gamma_{j}:[0,\epsilon]\longrightarrow 
 	\Omega$, $j=1,2$, 
 	such that $x_0=\gamma_1(0)=\gamma_2(0)$ and the angle $\angle \gamma_1(s)x_0\gamma_2(s) \to 0$ as $s\to 0$.
 	
 	\begin{figure}[ht]
 		\begin{center}
 			\includegraphics[scale=1.2]{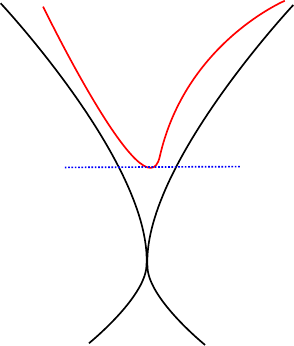}
 			\put(-90,95){$y_\delta$}\put(-100,55){$x_0$}
 			\put(-110,95){$p_\delta$}\put(-70,95){$q_\delta$}
 			\caption{Construction in the proof of Morse function case.}
 			\label{p}
 		\end{center}
 	\end{figure}
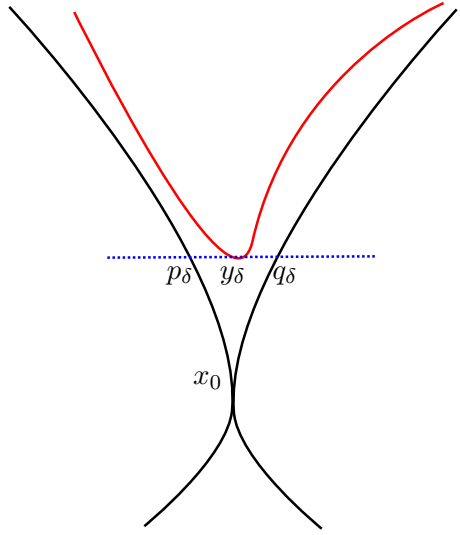
 	
 	Consider a level set $\{p=\lambda^*+\delta\}$, for small $\delta>0$. Let $y_\delta\in \{p=\lambda^*+\delta\}$
 	be the point on $\{p=\lambda^*+\delta\}$ closest to $x_0$, i.e. $|x_0-y_\delta|=\min_{y\in \{p=\lambda^*+\delta\}}|x_0-y|$,
 	and we proved that $|x_0-y_\delta|=O(\sqrt\delta)$. Let $p_\delta$, $q_\delta$ be the intersection between
 	$\{p=\lambda^*\}$ and the tangent line to $\{p=\lambda^*+\delta\}$ through $y_\delta$.
 	Clearly, as $\{p=\lambda^*\}$ has a cusp at $x_0$, we get $\lim_{\delta\to 0}\angle p_\delta x_0 q_\delta=0$.
 	Thus 
 	$$d(y_\delta,\{p=\lambda^*\})\le |p_\delta - y_\delta|=o(\sqrt\delta). $$
 	Since $\nabla p(x_0)=0$, and the gradient $\nabla p$ is $L$-Lipschitz continuous for some constant $L$, we infer
 	$|\nabla p(y)|\le A\sqrt\delta$ for some $A>0$ and all $y$ on the segment $[p_\delta,y_\delta]$.
 	Thus it follows
 	\begin{equation*}
 	\delta=|p(p_\delta)-p(y_\delta)|\le AL\sqrt\delta|p_\delta - y_\delta|=o(\delta).
 	\end{equation*}
 	This is a contradiction.
 	\\
 	\\
 	{\bf Proof of C2.} Let $U=\{p\ge \lambda^*+\delta\}$. Fix an arbitrary $r$. Clearly $U\subseteq \bigcup_{x\in U}B(x,r/3)$.
 	Since $U=\{p\ge \lambda^*+\delta\}$ is closed, and the domain $\Omega$ is compact, we infer $U=\{p\ge \lambda^*+\delta\}$
 	is also compact. Thus we can extract a covering $U\subseteq \bigcup_{i=1}^{C_r}B(x_{i},r/3)$ with finitely many balls. By Vitali
 	covering lemma, we can further extract mutually disjoint balls $B(x_{i_j},r/3)$ such that
 	$$U\subseteq \bigcup_{j=1}^{C_r'}B(x_{i_j},r)$$
 	Since $B(x_{i_j},r/3) \subset \Omega_{r/3}  $, and $\{B(x_{i_j},r/3)\}_{j=1}^{C_r'}$ are pairwise disjoint,
 	we have
 	$$V_dC_r' (r/3)^{-d}\le \mathcal L^d (\Omega_{r/3}).$$ 
 	Thus we can choose $\mathcal{N}_r=\{x_{i_j}\}$, $j=1,\cdots,C_r'$.

  \end{proof}

\subsubsection{Proofs in \Cref{subsection:flat-lowerbound}}

\begin{proof}[Proof of lemma \ref{lemma:Holder-lower-bound}]
	Let  $\lambda>0$ be given. Later in the proof, it can be seen that
	$\lambda =C_0 $, being the common upper bound of $f_i\in F$. 
	Define  $a>0$ to be such that 
	\begin{equation}\label{eq:size-of-domain}
		56\lambda \cdot 8^{d-1}a^d=1.
	\end{equation} 
	Consider 
	\begin{equation}
		\label{eq:base-function}
		f(x)=\begin{cases}\lambda  ,\ \ x\in[0, 56a]\times [0,8a]^{d-1}=\Omega\\
			0, \ \ \text{otherwise.}
		\end{cases}
	\end{equation}
	Let $b=\left(\frac{\log(32)}{n\lambda V_d}\right)^{1/d}$.
	For $0<\alpha < 1$, define
	$$g(r)=\begin{cases}
	0, \ \ 0\le r\le b\\
	\K  \left(\dk-|r-b-\dk |\right)^{\alpha}, \ \   |r-b-\dk|\le\dk\\ 
	0, \ \  \text{otherwise},\\
	\end{cases} $$
	and for $\alpha\ge 1$,define
	$$g(r)=\begin{cases}
	0, \ \ 0\le r\le b\\
	2^{1-\alpha}\delta-\K|r-b-\dk|^{\alpha}, \ \  0\le |r-b-\dk|\le\frac{1}{2}\dk\\ 
	\left(\delta^{1/\alpha} -\K^{1/\alpha}|r-b-\dk|\right)^{\alpha}, \ \frac{1}{2}\dk \le |r-b-\dk|\le \dk\\ 
	0, \ \  \text{otherwise},\\
	\end{cases} $$
	where $\K\le 1$ is chosen so that $g\in\Sigma (L,\alpha)$. By construction  $0\le g(r)\le  \delta$. \\
	\begin{figure}[ht]
		\center
		\includegraphics[width=\columnwidth]{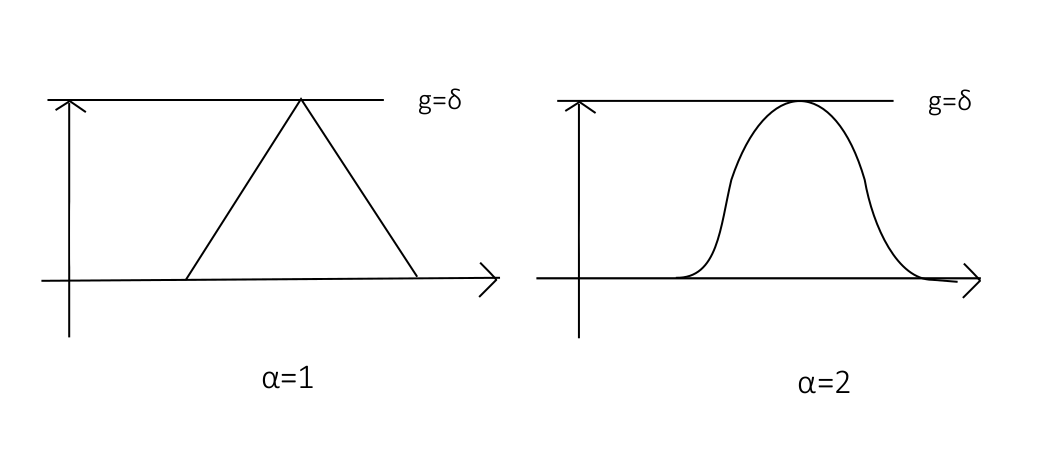}
		\caption{The radial function $g$ for $\alpha=1 ,2$. }
		\label{fig:galpha}
	\end{figure}
	Consider the inequality 
	\begin{equation}
		\begin{aligned}
			b+2\dddelta & =\left(\frac{\log(32)}{n\lambda V_d}\right)^{1/d} + 2\dddelta \\
			&\le \left(\frac{1}{4^d 8\lambda }\right)^{1/d} +2\left(\frac{1}{16(7\lambda)^{1/d}}\right) \\
			& \le \frac{1}{4} \left(\frac{1}{7}\right)^{1/d}\lambda^{-1/d} +\frac{1}{8}\left(\frac{1}{7}\right)^{1/d}\lambda^{-1/d}= 3a
		\end{aligned}
	\end{equation} 
	where the first inequality follows from $ n\ge 4^d\frac{8\log(32)}{V_d}$ and  $\delta\le  \left( \frac{\K}{16^{\alpha}(7\lambda)^{\alpha/d} }\right)$. 
	So by construction, $9$ disjoint ball  $\left\{B\left(x_i,b+2\dddelta\right)\right\}_{i=0}^8 $ can be placed in
	$\Omega$.
	For $i=1,\ldots,8$, let $f_i= f(x)-g(|x-x_i|) +g(|x-x_0|)$. Thus $f_i\in \Sigma(L,\alpha)$ for $,i=1,\ldots,8$ and the common upper bound of $f_i$ is  $\|f_i\|_{\infty}= \lambda$.
	Since $\int f =1$, by symmetry each $f_i$ also integrates to 1.
	The fact that $f_i\ge 0$ follows from  $0\le g(r)\le \delta \le p_{\max}$.  
	Since for any $1\le i,j\le 8$ ,  
	$f_j(x) = \lambda$ for any  $x\in B(x_i,b)$.
	\begin{equation*}
		\begin{aligned}
			P(\text{There exists a point in } B(x_i,b) \text{ for any } i )
			&\ge 1-(1-\lambda V_db^d)^n\\
			&\ge 1-\exp(-V_db^d\lambda n)= 1-1/32
		\end{aligned}
	\end{equation*}
	where  $b=\left(\frac{\log(32)}{n\lambda V_d}\right)^{1/d}$ is used in the last equality.
	Thus $$P(\text{There exists a point in every \ $B(x_i,b)$\  for } \ i =1 \ldots 8) \ge 3/4.$$ 
	\\
	\\
	Suppose the family $F=\{f_i \}_{i=1}^8 $ is given ahead. One wants to show  that any  algorithm being $\delta$ consistent with probability $3/4$ can identify $f_i$  with probability at least $1/2$.
	To begin consider $B_i=\{f_i\ge \lambda\}$. $B_i$ has exactly two connected components and one is $B(x_i,b)$. Denote the other connected component of $B_i$ by $V_i $. Thus $B_i=V_i\cup B(x_i,b)$, where $V_i \cap B(x_i,b)=\emptyset$. 
	Define the three events $\mathcal E_1$, $\mathcal E_1$ and $\mathcal E_3$ as following
	\begin{equation}
		\begin{aligned}
			\label{eq:events}
			&\mathcal E_1 =\{ \text{There exists a point in every \ $B(x_i,b)$\  for } \ i =1 \ldots 8 \} \\
			&\mathcal E_2 =\{ \text{The algorithm is  $(\delta, \epsilon)$ consistent  }\} \\
			&\mathcal E_3 =\{ \text{The algorithm can indentify the true density} \} 
		\end{aligned}
	\end{equation}
	Then one has $\mathcal E_1 \cap \mathcal E_2 \subset \mathcal E_3 $. This is 
	because if an algorithm is $\delta,$ consistent  and every $B(x_i,b)$ contains at least one point ,
	the algorithm will assign points in $\cup_{j\not =i} B(x_j,b)$ and points $B(x_i,b)$ into different clusters before joining them into the same cluster. 
	In this way, the algorithm can identify the true density. Since 
	$ P(\mathcal E_1)\ge 3/4 $ and $ P(\mathcal E_2)\ge 3/4 $,  $ P(\mathcal E_3)\ge 1/2 $
	\\
	It remains to compute the KL divergent between $f_1$ and $f_2$  and apply Fano's lemma.
	Using spherical coordinate centering at $x_1$ and $x_2$, the KL divergent is given by 

		\begin{align*}
			\text{KL}(f_1,f_2) &=dV_d\int _{b}^{b+2\dk}(\lambda) \log\left(\frac{\lambda}{\lambda -g(r)}\right)r^{d-1}+ (\lambda -g(r)) \log\left(\frac{\lambda-g(r)}{\lambda }\right)r^{d-1} dr\\
			&= dV_d \int _{b}^{b+2\dk}g(r) \log\left( \frac{\lambda }{\lambda -g(r)}\right)r^{d-1} dr\\
			&=  dV_d\int _{b}^{b+2\dk} g(r) \log\left( 1+\frac{g(r) }{\lambda -g(r)}\right)r^{d-1} dr\\
			&\le dV_d \int _{b}^{b+2\dk} g(r)\frac{g(r)}{\lambda-g(r)}r^{d-1} dr\le dV_d \int _{b}^{b+2\dk }g(r)\frac{g(r)}{\lambda}r^{d-1} dr\\
			&\le d\lambda^{-1}\delta^2V_ d \int _{b}^{b+2\dk }r^{d-1} dr\\
			&\le \frac{d\delta^2V_ d}{\lambda d} \left(b+2\dk\right)^{d} \\
		\end{align*}

	Thus by Fano's lemma 
	\begin{equation}
		\begin{aligned}
			n\ge \frac{(1/2)\log_2(8) -1}{\text{KL}(f_1,f_2)} = \frac{1}{2\text{KL}(f_1,f_2)}
		\end{aligned}
	\end{equation}
	and this implies 
	\begin{equation}\label{eq:312-1}
		\begin{aligned}
			\left(\frac{\lambda }{2\delta^2V_dn}\right)^{1/d}
			\le 2\dddelta+\left(\frac{\log(32)}{n\lambda V_d}\right)^{1/d}.
		\end{aligned}
	\end{equation}
	Since $\delta\le \lambda/(2^{d/2 +1})$, this gives 
	\begin{equation}\label{eq:313-1}
		\begin{aligned}
			\frac{\lambda }{2^{d+1}\delta^2 V_dn }\ge \frac{\log(32)}{n\lambda V_d}.
		\end{aligned}
	\end{equation}
	Combines equation (\ref{eq:312-1}) and equation (\ref{eq:313-1}) one has
	\begin{equation}
		\begin{aligned}
			2\dk&\ge 
			\left(\frac{\lambda }{2\delta^2V_dn}\right)^{1/d}(1-\frac{1}{2})
		\end{aligned}
	\end{equation}
	This gives 
	\begin{equation}
		\begin{aligned}
			n\ge \frac{\lambda \K^{d/\alpha}}{C(d)\delta^{2+d/\alpha}},
		\end{aligned}
	\end{equation}
	where $C(d) =2^{2d+1}V_d$. 
	\\
	\\
	To justify that the collection of functions $\{f_i\}_{i=1}^8$ constructed in the previous proof satisfies condition {\bf C} and {\bf S$(\alpha)$}, observe that $\lambda $ is the only split level of $f_i$ for all $1\le i\le 8$. The case of  $\alpha>1 $ is  only provided as the case of $\alpha<1$ is simpler. Straight forward computations shows that 
	for any  $t\le 2^{-\alpha}$,
	$ \{x: f_i(x)  \ge \lambda+t\}$ has two connected components:  $ B \left(x_i, b_0 +\frac{1}{2} \left( \frac{\delta}{\mathcal K}\right)^{1/\alpha} -\left( \frac{t}{\mathcal K}\right)^{1/\alpha}  \right) $ and $  \left( B \left(x_i, b_0 +\frac{1}{2} \left( \frac{\delta}{\mathcal K}\right)^{1/\alpha}  \right) \right)^c   \cap \Omega.  $
	Therefore condition  {\bf C} and {\bf S$(\alpha)$} are trivially satisfied. 
\end{proof}

\subsection{Proofs in \Cref{section:consistent split level}}

\begin{proof}[Poof proposition  \ref{prop:find split level}] 
	By \eqref{eq:KDE risk}, with probability at least $\gamma$, 
	we have   
	$$
	\|\widehat p_h-p \|_{\infty}\le a_n.
	$$
	{\bf Step 1.} In this step, we show that for any split level $\lambda^*$ satisfying 
	\eqref{eq:split regular}, there exists $\hatsplit $ being 
	$\Delta$-significant and that 
	$$|\hatsplit -\lambda^* |\le \Delta$$
	for large $n$.
	Let $\C$ and $\C'$ be two sets split at $\lambda^*$.
	Thus there exists $\mathcal B$ being the connected component of $\{p\ge \lambda^*\}$ containing both $\C$ and $\C'$. \\
	By \eqref{eq:split regular}, for large n,
	neither  $\{X_i\}_{i=1}^n\cap \C\cap \{p\ge \lambda^*+2\Delta\}$ nor 
	$\{X_i\}_{i=1}^n\cap \C'\cap \{p\ge \lambda^*+2\Delta\}$ is empty.
	Let $X_i\in \C\cap \{p\ge \lambda^*+2\Delta\}$ and $X_j\in\C'\cap \{p\ge \lambda^*+2\Delta\}$.
	\\
	\\
	$\bullet$ By the same argument  that gives 
	\eqref{eq:C MDBSCAN}	\begin{equation}\label{eq:consistency of split level}
	\{p\ge \lambda^*\}\subset  \widehat L( \lambda^* -a_n) : = \bigcup_{\{ X_j :  \widehat D( \lambda^* -a_n) \}} B(X_j, 2h).
	\end{equation}
	Since $\mathcal B\subset\{p\ge \lambda^*\}$ and that $\mathcal B$ is connected, $X_i$ and $X_j$ have the same label in $ \mathbb C( h, \lambda^*-a_n)$. 
	\\
	\\
	$\bullet$  Since $X_i \in \mathcal C, X_j \in \mathcal C'$ and that $\mathcal C$ and $\mathcal C' $
	are split exactly at $\lambda^*$, $X_i, X_j$ are   contained in the distinct connected components of $\{ p\ge \lambda^* +\Delta\} $. 
	By {\bf Claim 2} in the proof of \Cref{prop:MDBSCAN}, $X_i$ and $X_j$ belong to distinct connected components of $ \mathbb C(h, \lambda^* +\Delta -a_n). $
	Let
	$\hatsplit$ be defined as in \eqref{eq:find-split-hat}.
	By the above two bullet points, 
	$$\lambda^*-a_n \le \hatsplit \le \lambda^*+\Delta-a_n.$$ 
	The fact that  $\hatsplit $ is $\Delta$-significant follows from the observation that 
	$$ X_i,X_j\in  \mathbb C (h, \lambda^* +2\Delta-a_n).
	$$
	\
	\\
	{\bf Step 2.} In this step, we show that if $\hatsplit$ is a $\Delta$-significant level of the cluster tree constructed using  modified DBSCAN, then there exists 
	$ \lambda^*$ being a split level of $p$ such that
	$$|\hatsplit-\lambda^* |\le \Delta .$$
	So suppose  $X_i$ ,$X_j$ and $\hatsplit $ satisfies \eqref{eq:find-split-hat} and that 
	$X_i,X_j\in \mathbb C (h, \widehat {\lambda^*}  +\Delta)$.
	Let
	$$\lambda^* :=  \sup\{\lambda\ge 0: X_i \text{ and } X_j \text{ are in the same
		connected component of }
	\{ p\ge \lambda\} \}.    $$
	$\bullet$ By  \eqref{eq:consistency of split level}, $X_i$ and $X_j$ have the same label in $ \mathbb C( h, \lambda^*-a_n)$.  Therefore, 
	$$ \lambda^* -a_n \le \hatsplit. $$
	\
	\\
	$\bullet$  For the sake of contradiction, suppose that 
	$$ \hatsplit > \lambda^* +\Delta. $$
	Then by {\bf Claim 2} in the proof of \Cref{prop:MDBSCAN}, $X_i$ and $X_j$ belong to distinct connected components of $ \mathbb C(h, \lambda^* +\Delta -a_n). $
	By definition of $\hatsplit$, this implies 
	$$ \lambda^* +\Delta-a_n \ge \hatsplit ,$$
	which is a contradiction.
	This finishes the proof.

\end{proof}

\subsection{A Side result: consistency of the KDE tree}
\label{section:KDE}
As a side result, we also compute the upper bound of cluster tree estimators generated by kernel density estimators. 
We acknowledge that KDE clustering algorithms  have been studied by many authors including 
\cite{rinaldo2010generalized}, 
\cite{rigollet2009optimal} and \cite{inference.tree}.   For completeness, we provide  $\delta$-consistency results for KDE cluster tree returned by \Cref{algorithm:clustering cc}, but we do not claim the novelty of these results.

\begin{algorithm}[!ht]
	\begin{algorithmic}
		\INPUT i.i.d sample $\{X_i\}_{i=1}^n$,  the kernel $K: \mathbb R^d \to \mathbb R$, the level $\lambda$ and $h>0$
		\State 1. Compute $\widehat L(\lambda) = \{x: \widehat p_h(x)\ge \lambda\} $.
		\State 2. Construct a graph  $\mathbb G_{h,k}$ with nodes 
		$$\widehat D(\lambda)=\{X_i\}_{i=1}^n\cap \widehat L(\lambda)$$ and edges $(X_i, X_j)$  if  
		$X_i $ and $X_j$ belong to the same connected component of $\widehat L(\lambda)$. 
		\State 3. Compute $\mathbb C(h,\lambda)$, the graphical connected components of $\mathbb G_{h,\lambda}$.
		\OUTPUT $\widehat{T}_n = \{ \mathbb C(h,\lambda) , \lambda \geq 0 \}$.
		\caption{Clustering based on connected components}
		\label{algorithm:clustering cc}
	\end{algorithmic}
\end{algorithm}
We start by showing that for generic $\alpha>0$, if $p \in \Sigma(L,\alpha)$, level sets of  KDE estimator are good approximations of 
the corresponding population quantities. 
\begin{lemma}\label{lemma:standard-level-inequality}
	Assume that $p \in \Sigma(L,\alpha)$, where $\alpha > 0$,  and let $K$ be a $\alpha$-valid kernel.
	Then, there exist constants $C_1$ and $C_2$, depending on
	$\|p\|_{\infty}$, $K$,  $L$ and $d$   such that if $h=C_1\frac{1}{n^{1/(2\alpha+d)}}   $, then
	with probability $1-1/n$, uniformly over all $\lambda>0$,
	\begin{equation}
	\begin{aligned}
	&\left\{x \colon  p(x) \ge \lambda +C_2 \left( \frac{\log(n)}{n} \right)^{\alpha/(2\alpha+d)}\right \}  \subset
	\{ x \colon \widehat p_h(x) \ge \lambda \}\subset \left\{ x \colon p(x) \ge \lambda
	-C_2 \left( \frac{\log(n)}{n} \right)^{\alpha/(2\alpha+d)} \right \}.   \\
	\end{aligned}
	\end{equation}
\end{lemma}

As a direct corollary of \Cref{lemma:standard-level-inequality}, we show that algorithm \ref{algorithm:clustering cc} is consistent with the optimal rate.
\begin{corollary} \label{coro:flat-upper}
	Let $h$ be chosen as in \Cref{lemma:standard-level-inequality}.	Under the assumptions of \Cref{lemma:standard-level-inequality}, the cluster tree returned by 
	\Cref{algorithm:clustering cc}  
	is
	$\delta$-consistent with probability at least $1-1/n$, where 
	\begin{equation}\label{eq:flat-upper}
	\delta\ge 3C_2 \left( \frac{\log n}{n} \right)^{\alpha/(2\alpha+d)},
	\end{equation}
	with  $C_2=C_2( \|p\|_{\infty},K,L,d)$ a constant independent of $n$ and $\delta$.
\end{corollary}
We remark that  \Cref{algorithm:clustering cc} is computationally  infeasible even  in small dimensions. This is mainly because it requires to compute   the level set  $\{x: \widehat p_h(x)\ge \lambda\}$ exactly to determine the clustering structure of the data points.  However  \Cref{algorithm:clustering cc}  does not require additional regularity conditions such as {\bf S$(\alpha)$} and {\bf C} to attain the minimax optimal rates.

 \subsubsection{Proofs in \Cref{section:KDE}}
 
 \begin{proof}[Proof of lemma \ref{lemma:standard-level-inequality}]

 	For any $x\in \mathbb R^d$,  with probability at least $1-1/n$
 	\begin{equation}\label{eq:bias-variance}
 	\begin{aligned}
 	|\hat p_h(x) -p(x) |
 	\le& |\hat p_h(x) -p_h(x) | +|p_h(x) -p(x)|\\
 	\le&   C_1(K, d,\|p\|_{\infty}) \sqrt{\frac{  \log n }{nh^d }}+ C_2(K,\alpha,L) h^\alpha\\
 	\end{aligned}
 	\end{equation}
 	where the second inequality follows from proposition \ref{prop:Steinward} and standard calculations for the bias. By taking 
 	$$h=h_n=\Theta\left(\frac{\log n}{n}\right)^{1/(2\alpha +d)}$$
 	in \eqref{eq:bias-variance},
 	$$\sup_{x\in \mathbb R^d} |\hat p_h(x) -p(x) | \le C( \| p\|_{\infty},K,L,\alpha ,d) \left(\frac{\log(n)}{n}\right)^{\alpha/(2\alpha +d)}$$
 	This completes the proof.
 	%\item [Step 2.]It remains to show that 
 	%if $x \in \Omega\backslash \Omega_{-h}$, $\hat p_h(x) \ge \lambda$ implies $ p(x) \ge \lambda -C\log(n) \frac{\K(L)^{d/(2\alpha+d)}}{n^{\alpha/(2\alpha+d)}}$.  Suppose for the sake of contradiction that 
 	%  $ p(x) < \lambda-C\sqrt{\frac{\log(h)}{nh^d} }- \K(L) h^\alpha.$ Let $\tilde p$ be the extension of $p$ in  definition \ref{Holder-continuous}. Then 
 	%  $p(x)=\tilde p(x)< \lambda-C\sqrt{\frac{\log(h)}{nh^d} }- \K(L) h^\alpha.$ Since
 	%   $| \tilde p_h-\tilde p|\le \K(L)h^{\alpha}$,
 	%  $$\tilde p_h(x) < \lambda -C\sqrt{\frac{\log(h)}{nh^d} }$$
 	%Since $$p_h(x) =\int_{\mathbb R^n}\frac{1}{h^d} K\left(\frac{x-y}{h}\right)p(y)dy\le \int_{\mathbb R^n}\frac{1}{h^d} K\left(\frac{x-y}{h}\right) \tilde p(y)dy\le \tilde p_h(x),$$
 	%$p_h(x) < \lambda -C\sqrt{\frac{\log(h)}{nh^d} }$. This implies $\hat p_h(x)< \lambda$, a contradiction.
 \end{proof}
 \
 \newline
 \newline
 
 \begin{proof}[Proof of Corollary \ref{coro:flat-upper}]
 	Let $A$ and $A'$ be two given  connected subsets of $\mathbb R^d$. Suppose $\lambda>0$ satisfies  
 	$\lambda+3\delta =\inf_{x\in  A\cup A'} f(x) $ and that 
 	$A$ and $A'$ are  contained in two distinct connected components of $\{ p> \lambda\}$.
 	It suffices to show that the estimate cluster tree at 
 	$\{\hat p_h\ge \lambda+2\delta\}$ gives  correct labels to $A\cap\{X_i\}_{i=1}^n$ and $A'\cap\{X_i\}_{i=1}^n$,
 	where 
 	$$ h=h_n=\Theta\left(\frac{\log n}{n} \right)^{1/(2\alpha+d)}$$
 	\begin{itemize}
 		\item Since $A$, $A'$ are connected and
 		$$ A,\  A'\subset \{p\ge 3\delta+\lambda\}\subset \{ \hat p_h\ge 2\delta+\lambda\},$$
 		$A$ and $A'$ each belongs to the connected component of $\{\hat p_h\ge 2\delta+\lambda\}$. Therefore the cluster tree at  $\{\hat p_h\ge 2\delta+\lambda\}$ will assign $A\cap \{X_i\}_{i=1}^n$ the same label. This is also true for $A'\cap \{X_i\}_{i=1}^n$.
 		\item It remains  to show that  $A$ and $A'$ are in the two distinct connected components of $\{\hat p_h\ge 2\delta+\lambda\} $. 
 		For the sake of contradiction, suppose that $A$ and $A'$ are in the same connected components of $\{\hat p_h\ge 2\delta+\lambda\} $.
 		Since 
 		$$\{ \hat p_h\ge \lambda+2\delta \}\subset \{p\ge \lambda+\delta \}\subset \{p> \lambda\} ,$$
 		$A$ and $A'$ are in the same  connected components of $\{p> \lambda\} $. This is a contradiction.
 	\end{itemize}
 \end{proof}

\section{Proofs from Section \ref{section:gap}}

\subsection{Proof of \Cref{prop:consistency-gap}}
We first prove two simple technical lemmas, that will also be used in the proof of \Cref{prop:lower gap level set}. % \ref{lemma:inclusionofgap} and \ref{lemma:inclusionofgap2}.

\begin{lemma} \label{lemma:inclusionofgap}
	Suppose that $\epsilon > 2a_n$, where 
	\[
	\sup_{x\in \mathbb R^d} |\hat p_h(x)-p_h(x)|\le a_n,
	\]
	and let 
	$\lambda \in (\lambda_*+a_n,\lambda^* -a_n)$.
	Then,
	$$ S_{-h}\cap \{X_i\}_{i=1}^n  \subset \hat D_h(\lambda)\subset S_{h} , $$ 
	where $S=\{p\ge \lambda^* \}$ and
	$$\hat D_h(\lambda)= \left\{x: \hat p_h(x)\ge \lambda \right\} \ \bigcap \ \{X_i\}_{i=1}^n.$$
\end{lemma}

\begin{proof}[Proof of lemma \ref{lemma:inclusionofgap}]\
	For the first inclusion, suppose $X_j\in S_{-h}  \cap  \{X_i\}_{i=1}^n.$
	Then $B(X_j,h)\subset S $. Since $K$ is supported on $B(0,1)$,
	\begin{equation}\label{eq:Klambda}
		p_h(X_j)=\frac{1}{V_dh^d}\int_{B(X_j,h)} p(y) dy\ge \lambda^* .
	\end{equation}
	As a result, $$\hat p_h(X_j)\ge  p_h(X_j) - a_n \ge \lambda^*-a_n \ge\lambda,$$
	which implies that $ X_j\in \hat D_h(\lambda)$.
	For the second inclusion, if $X_j \in \hat D_h(\lambda)$, then 
	$\hat p_h(X_j)\ge \lambda$.
	So $$p_h(X_j)\ge \hat{p}_h(X_j) - a_n \geq  \lambda-a_n>\lambda_*$$
	However, for any point $x\in S_{h}^c$, since  $B(x,h)\subset S^c$,
	$p_h(x)\le\lambda_*$ (see \eqref{eq:Klambda}). 
	So $X_j\in\hat D_h(\lambda)$ implies $X_j \in S_h$.
\end{proof}
\begin{lemma}\label{lemma:inclusionofgap2}
	Under the same assumption as in \Cref{lemma:inclusionofgap}, suppose further
	that $\lambda^*>a_n$.
	Let $\hat L (\lambda)= \bigcup_{X_i\in \hat D_h(\lambda)} B(X_i,h)$
	and $\mathcal C$ be any connected components of $S$.
	Then $\mathcal C_{-2h} \subset \hat L(\lambda).$
\end{lemma}

\begin{proof}[Proof of lemma \ref{lemma:inclusionofgap2}]
	Let $x\in \mathcal C_{-2h}$. Then, $B(x,h)\subset S$, which implies, by
	\eqref{eq:Klambda}, that  
	$p_h(x)\ge \lambda^*$ and therefore that
	$$\hat p_h(x) \ge  p_h(x) - a_n \geq \lambda^*-a_n>0 .$$
	Therefore,  $B(x,h)\cap\{X_i\}_{i=1}^n$ is not empty -- otherwise $\hat p_h(x)=0 $ -- so
	that there exists a sample point, say $X_j$,
	in $B(x,h)$.
	Since $B(x,h)\subset S_{-h}$, we conclude that $X_j\in S_{-h}$. By lemma
	\ref{lemma:inclusionofgap} we then have that $X_j\in \hat D_h(\lambda)$.
	This shows that if $x\in \mathcal C_{-2h}$, then there exists some $X_j \in \hat D_h(\lambda)$ such that $x\in B(X_j,h)$. This finishes the lemma.
\end{proof}

\begin{proof}[Proof of \Cref{prop:consistency-gap}] 
Let
$$a_n  =C_1 \sqrt{\frac{\log(n) +\log (1/ h)}{nh^d}} $$  be  defined as in \eqref{eq:an}.
Then by \eqref{eq:variance of density},
\begin{equation}
	P\left(\sup_{x\in \mathbb R^d} |\hat p_h(x)-p_h(x)|\le a_n \right )\ge 1-1/n.
\end{equation}

Denote $h=C(\frac{\log (n) }{n\epsilon^2})^{1/d}$ where $C$ is chosen such that 
$3a_n\le \epsilon $. 
	\begin{itemize} 
		\item [Step 1.] Suppose $A_{2h}\subset \mathcal C_{i}$. Then $A\subset \mathcal C_{i,-2h} $.
		Then by \Cref{lemma:inclusionofgap2}, one has 
		$A\subset \mathcal C_{i,-2h}\subset \hat L(\lambda)$.
		Since data points in connected components of $\hat L(\lambda)$ have the same label, and $A$ is contained in only one connected components of 
		$\hat L(\lambda)$, points in  $A\cap \{X_i\}_{i=1}^n$ have the same labels.
		\item[Step 2.]Suppose A1 holds.
		Since $d(\mathcal C_{i} ,\mathcal C_{j})>4h$, $\{\mathcal C_{i,2h}\}_{i=1}^I$ are pairwise disjoint.\\
		Since $\hat D_h(\lambda_k)\subset \bigcup_{i=1}^I \mathcal C_{i,h}$, this means for any $i,j$ there is no edges connect $\hat D_h(\lambda_k)\cap \mathcal C_{i,h}$ and 
		$\hat D_h(\lambda_k)\cap \mathcal C_{j,h} $. 
		Since $A$ and $A'$ belong to  distinct members of $\{\mathcal C_i\}_{i=1}^I$, labels in $A\cap\{X_i\}_{i=1}^n$ and  in $A'\cap \{X_i\}_{i=1}^n$ are different. 
	\end{itemize}
\end{proof}

\subsection{Proof of  \Cref{prop:support estimate upper} }
%To show \Cref{prop:support estimate upper}, 

\begin{proof}[Proof of \Cref{prop:support estimate upper}]
Let
$$a_n  =C_1 \sqrt{\frac{\log(n)+\log (1/ h)}{nh^d}} $$  be  defined in \eqref{eq:KDE risk}.
Then by \eqref{eq:variance of density}.
	\begin{equation}
		P\left(\sup_{x\in \mathbb R^d} |\hat p_h(x)-p_h(x)|\le a_n \right )\ge 1-1/n.
	\end{equation}
	Denote $h=C(\frac{\log(n) }{n\epsilon^2})^{1/d}$ where $C$ is chosen such that 
	$3a_n\le \epsilon $. 
	Denote $\lambda_k= \frac{k}{nh^dV_d}$.
 Therefore 
	$\lambda^*-\lambda_*\ge 3a_n .$
	\\
	Consequently $\lambda$ is well defined and one has
	$$\lambda_* +a_n\le \lambda < \lambda^*-a_n. $$
	By lemma \ref{lemma:inclusionofgap}, the nodes
	of $\mathbb G_{h,k} $ are contained in $ S_h$. Thus
	\begin{equation}\label{eq:support upper}
		\bigcup_{X_j \in\mathbb G_{h,k} } B(X_j,h) \subset S_{2h}.
	\end{equation}
	By lemma \ref{lemma:inclusionofgap2}, 
	\begin{equation}\label{eq:support lower}
		S_{-2h} \subset\bigcup_{X_j \in \mathbb G_{h,k}} B(X_j,h).
	\end{equation}
	Since $h_0\ge h $ using  assumption A3 then 
	$$\mathcal L(\hat S \triangle S) \le \mathcal L(S_{2h} \backslash S) 
	+\mathcal L(S_h \backslash S_{-2h} )\le C_0 h .$$
\end{proof}

\subsection{Proof of \Cref{prop:lower gap level set}}
\label{sec:level-set-gap}

We begin by constructing of a well-behaved class
of sets  satisfying the boundary
regularity condition {\bf (R)}. These sets will then be used to define high-density
clusters in the proof of \Cref{prop:lower gap level set}. Sets satisfying the properties
given in the next definition are well known in the literature on support
estimation: see, e.g.,   
\cite{korostelev2012minimax}. For completeness we also show that 
they satisfy the boundary
regularity condition {\bf (R)}.
\begin{definition}\label{def:domains-reg}
	Denote  by $\mathcal G_{d}(L)$ the class of all domains in $[0,1]^d$ satisfying 
	$$\left\{(x_1 ,\ldots, x_d) : (x_1,\ldots, x_{d-1})\in [0,1]^{d-1}, \ \  0\le x_d \le g(x_1,\ldots,x_{d-1})  \right\},$$
	where $g:\mathbb R^{d-1} \to \mathbb R$ satisfies
	\begin{itemize}
		\item $ 1/2\le |g(x)|\le 3/2 $ for all $x\in [0,1]^{d-1}$
		\item $| g(x)-g(x')| \le L|x-x'|$ for all $x,x'\in \mathbb R^{d-1}$.
	\end{itemize}
	
\end{definition}

\begin{lemma}
	\label{lemma:preserve distance volume}
	There exist a  constants $h_0$ only depending only on $L$   such that 
	for any $\Omega \in \mathcal G_{d}(L) $, one has for any $0\le h\le h_0$,
	$$\mathcal L(\Omega_{h} \backslash \Omega_{-h}) \le C_0 h,$$
	where $\Omega_{h}=\bigcup_{x\in\Omega} B(x,h) $, $\Omega_{-h} = \{x\in \Omega: B(x,h)\subset\Omega\}$ and $C_0$ is some constant depending on $d$.
\end{lemma}
\begin{proof}[Proof of lemma \ref{lemma:preserve distance volume}]
	Given $\Omega \in \mathcal G_d(L) $, let $g$ be the corresponding map
	as in definition \ref{def:domains-reg}. Denote $\underline x$ be a generic point in $\mathbb R^{d-1}$.
	Consider the change of coordinate map $\phi : \mathbb R^d \to \mathbb R^d$ defined 
	as $$\phi (\underline x,x_d) =(\underline x, \  x_d g(\underline x)) .$$
	The inverse map $\phi^{-1}:\mathbb R^d \to \mathbb R^d$ where 
	$\phi^{-1} (\underline x,x_d) =(\underline x, \  x_d/ g(\underline x))$  is also well defined as $g>0$. \\
	Observe that $\phi([0,1]^{d})=\Omega$, and there exists a constant $C(d)$ depending only on $d$ such that $[0,1]^d$ satisfies condition A3 with $h_0=1/2$ and $C_0=C(d)$.
	Thus in order to justify the lemma, it suffices to show that the maps $\phi$ and $\phi^{-1}$
	only distort the distance and volume by factors depending on $L$ only.\\
	To be more precise, it suffices to show that for some constant $L'$ depending on $L$ and some absolute constant $C$,
	\begin{align*}
		|\phi^{-1}(x)-\phi^{-1}(x')|\le L'|x-x'| &\text{ and }|\phi(x)-\phi(x')|\le L'|x-x'| \text{ for all $x ,x' \in [-2,2]^d$} 
		\\
		\mathcal L( \phi^{-1}(B))\le C\mathcal L(B) &\text{ and } \mathcal L( \phi(B))\le C\mathcal L(B) \text{ for any $B\subset [-2,2]^d$ }.
	\end{align*} 
	Since the calculations
	of $\phi$ are similar to that  of $\phi^{-1}$, only the former one is shown in this case. 
	\begin{itemize}
		\item[Step 1.]To show that $\phi(x)$ is Lipschitz, it suffices to bound
		$\|\nabla \phi \|_{op}$. 
		\begin{align*}
			\nabla \phi(\underline x,x_d)=(\frac{\partial \phi_i}{\partial x_j})=
			\begin{bmatrix}
				1 & 0 & \dots & 0 & x_d\frac{\partial g (\underline x)}{\partial x_1} \\
				0 & 1 & \dots & 0 & x_d\frac{\partial g (\underline x)}{\partial x_2}   \\
				\vdots & \vdots & \ddots & \vdots  & \vdots \\
				0 & 0 & \dots & 1 & x_d\frac{\partial g (\underline x)}{\partial x_{d-1}}   \\
				0 & 0 & \dots & 0 & g(\underline x)   
			\end{bmatrix}
		\end{align*}
		A straight forward calculations shows that for any $(\underline x, x^d)\in[-2,2]^d $, $$\|\nabla \phi\|_{op}\le 1+x_0\|\nabla g(\underline x)\|_2+g(\underline x)\le 5/2+2L .$$ \\
		\item[Step 2.]The change of variables equations gives
		$$\mathcal L(\phi(B))= \int_{\phi (B)} d\mathcal L 
		=\int_{B}|\det (\nabla \phi (x))|dx.
		$$
		Since $\det (\nabla \phi (\underline x, x_d)) =g(\underline x)$ which is bounded above by  $3/2$, one has $ \mathcal L(\phi(B)) \le (3/2) \mathcal L(B)$.
	\end{itemize}
\end{proof}

\begin{proof}[Proof of \Cref{prop:lower gap level set}]
	
	Let $0<\delta\le 1/16$ be depending on $ \epsilon$ which will be specified later. For some constant depending $C(d) $ only depending on $ d$,
	it is desired to  construct $\{S_i\}_{i=1}^M\in\mathcal G_d(C(d))$ such that 
	$\mathcal L(S_i\triangle S_j) \ge \delta$ and that $M$ is of order $\delta^{-d+1}$.\\
	
	\begin{itemize}
		\item[Step 1.]Consider a hyper rectangle $[0,2\delta]\times [0,\delta]^{d-2} $  in $\mathbb R^{d-1}$. 
		One can place $N=\lfloor\delta^{-1}\rfloor^{d-1}/2$ such hyper rectangles into $[0,1]^{d-1} $ without having any two intersect.
		Denote theses hype rectangles by $\{R_i\}_{i=1}^N$. $R_i$ is composed of two hypercube of dimension
		$[0,\delta]^{d-1}$, which are denoted as $R_i^0$ and $R^1_i$. 
		\\ 
		\\
		Denote $\underline x$ to be a generic point in $\mathbb R^{d-1}$.
		One can  defined a map 
		$g:[-\delta/2,\delta/2]^{d-1} \to \mathbb R $ by
		$$
		g(\underline x)=\begin{cases}
		C(d)\left( \delta/2-\|\underline x\|_{\mathbb R^{d-1}}\right), \text { if }  \|\underline x\|_{\mathbb R^{d-1}}\le \delta/2 \\
		0, \text { otherwise .}
		\end{cases}
		$$
		The region
		$$\mathcal C = \{(\underline x, x_d) : \underline x\in [-\delta/2,\delta/2]^{d-1}, 0\le x_d\le g(\underline x)\} $$
		defines a region of hyper cone in $\mathbb R^d$ and $C(d) $ is set so that the cone volume  
		\begin{align*}
			\int_{[-\delta/2,\delta/2]^{d-1}} g(\underline x)d\underline  x=\delta^{d}.
		\end{align*}
		Let $g^0_{i}$ and $g^1_i$  be the corresponding map on $R_i^0$ and $R^1_i$, as the later ones are  copies of $[-\delta/2,\delta/2]^{d-1}$. 
		\\
		\item[Step 2.]Let $W =\{w =(w_1,\ldots, w_N) , w_j\in \{0,1\}\}$. By Varshamov-Gilbert lemma, there exist
		$w^1, \ldots,w^M \in W$ such that (i) $M\ge 2^{N/8}$ (ii) $H(w^{i} ,w^{j})\ge N/8 $, where $H$ denote the hamming distance. \\
		For $1\le j\le M $, let $G_j:[0,1]^{d-1} \to \mathbb R^d$ be defined as 
		$$G_j(\underline x) =1/2+\sum_{i=1}^N  g_i^{w^j_i}(\underline x).$$
		Consider 
		$$ S_j=\{(\underline x, x_d) : \underline x \in [0,1]^{d-1}, 0\le x_d\le G_j(\underline x) \}
		$$
		Thus by construction $G_j\in \mathcal G_d(C(d))$ in definition \ref{def:domains-reg}.
		For $l=0,1$ and $1\le i\le N$, define
		$$ \mathcal C_i^l=\{(\underline x,x_d) : \underline x \in R_i^l, 0\le x_d\le g^{l}_i(\underline x)\}.$$
		So $\mathcal C_j^l$ are non-overlapping cones with volume being $\delta^d$, which are indexical copies of  $\mathcal C$.
		
		\item[Step 3.] Let $\{f_j\}_{j=1}^M $ be such that 
		$$f_j=\begin{cases}
		1/4, \text{ if } x\in [0,a]^d\backslash S_j\\
		1/4+\epsilon, \text{ if } x\in S_j\\
		0, \text{ otherwise}
		\end{cases}
		$$
		Since $1=\int f_j =a^d/4+ (1/4+\epsilon)(3/4)^d \le a^d /4 + (1/2)(3/4)^d,$
		$a$ has to be greater than 1. 
		Thus $S_j\subset [0,a]^d $ and so $S_j$ can be viewed as the support of $f_i$ at the gap.
		\item[Step 5.] For any $i$ and $j$
		Since $f_i$ and $f_j$ are only possibly different on $\{\mathcal C_k^0 \cup\mathcal C_k^1 \}_{k=1}^N.$ Also $f_i\not = f_j $ within $\mathcal C_k^0 \cup\mathcal C_k^1 $
		if and only if $w^i_k\not = w^j_k $. Thus 
		the $KL(f_i,f_j)$ is determined by 
		\begin{align*}
			KL(f_i ,f_j) =\sum_{k=1}^N\int_{\mathcal C^0_k\cup \mathcal C^1_k} f_i\log\left(\frac{f_i}{f_j}\right) =\sum_{k: w^i_k \not = w^j_k }\int_{\mathcal C^0_k\cup \mathcal C^1_k} f_i\log\left(\frac{f_i}{f_j}\right)
		\end{align*}
		Suppose $w^i_k \not = w^j_k$, then
		\begin{align*}
			\int_{\mathcal C^0_k\cup \mathcal C^1_k} f_i\log\left(\frac{f_i}{f_j}\right)=
			\int_{\mathcal C} (1/4+\epsilon)\log\left(\frac{1/4+\epsilon}{1/4}\right) +(1/4)\log\left(\frac{1/4}{1/4+\epsilon}\right)  \le 4 \delta^d \epsilon^2.
		\end{align*}
		So $KL(f_i,f_j) \le H(w^i,w^j)\delta^d \epsilon^2 \le N \delta^d \epsilon^2$.\\
		\item[Step 6.] To apply Fano's lower bound lemma (see \cite{tsybakov2009introduction}) it suffices to have
		$$ \max_{i\not = j}KL(P_i, P_j)  \le \frac{\log M}{16n}
		$$
		Since $M\ge 2^N/8 $ it suffices to have 
		$ n N \delta^d \epsilon^2 \le N\log(2)/128$.
		Thus it suffices to have $\delta^d = \min \{a \frac{1}{n\epsilon^2},1/16\} $ for some absolute constant $a$.
		\item[Step 7.] By Fano's lemma, the minimax rate is bounded from below above by
		$$\mathcal L(S_i \triangle S_j)=H(w^i,w^j) 2\mathcal L(\mathcal C)\ge (N/8)2 \delta^d = c\delta,$$
		for some absolute constant $c$.
	\end{itemize}
\end{proof}

\subsection{Lower bounds of clustering at the Gap}
\label{section:lower bound clustering at the gap}
We cite Theorem
VI.1 of
\cite{chaudhuri2014consistent} to demonstrate that the scaling in \Cref{prop:consistency-gap} is   minimax optimal.

\begin{proposition}  \label{lemma:Dasgupta}
	Consider a finite family of density functions $F=\{f_j\}$. 
	Suppose all $f_j\in F$ have gap of size $\epsilon>0$ at level $\lambda_*$. This means that for any $j$, $\{f_j\ge\lambda_*+\epsilon \}\cup \{f_j\le\lambda_* \}=\mathbb R^d$. For any $j$, let $\{\mathcal C_j^i\}_{i=1}^{I_j}$ be the connected components of $\{f_j\ge \lambda_*+\epsilon\}$ and $d(\mathcal C^i_{j},\mathcal C^{i'}_{j})\ge\sigma$ for $i\not = i'$. \\
	There exists subsets $A_j $ and $A'_j$ for density $f_j$ such that $A_{j,\sigma}\subset \mathcal C^i_{j}$ and $A_{j,\sigma}' \subset \mathcal C^{i'}_{j}$ with the following additional property.\\
	Consider any algorithm that is given $n \ge 100$ i.i.d. samples $\{X_i\}_{i=1}^n$ from some 
	$f_j \in F$ and, with probability at least 3/4, outputs a tree in which the smallest cluster containing $A_j\cap \{X_i\}_{i=1}^n$ is disjoint from the smallest cluster containing $A_j' \cap \{X_i\}_{i=1}^n$. Then
	there exists a constant $C(d)$ only depending on $d$ such that
	\begin{equation}\label{eq:clusteer.minimax}
	n\ge \frac{C(d)}{\sigma^d\lambda^*\epsilon^2}\log\frac{1}{\sigma^d\lambda^*}.
	\end{equation}
\end{proposition}
The proof of the proposition can be found in  Theorem
VI.1 of
\cite{chaudhuri2014consistent}. We omit the
details of the proof for brevity.

  \section{A Cluster Consistency Resut for General Densities}
  \label{sec:general}
  The type of cluster consistency result we have obtained for densities with
  gaps can be easily
  generalized to arbitrary densities. To that end, we will introduce the notion
  of $(\epsilon,\sigma)$-separated clusters and of $h$-thick clusters, where
  $\epsilon$, $\sigma$ and $h$ are positive numbers..
 
  \begin{definition}
      Let $\epsilon$ and $\sigma$ be positive numbers. Two connected subsets $A$ and $A'$ of the support of $P$ are said to be $(\epsilon,
      \sigma)$-separated when
      \begin{itemize}
\item they belong to different connected components of $L(\lambda^* -
    \epsilon)$, where $\lambda^* =
\inf_{x \in A \cup A} f(x) > \epsilon $, and 
\item    $\min_{k \neq l} \mathrm{dist}(\mathcal{C}_k, \mathcal{C}_l) > \sigma$,
    where $\mathcal{C}_1,\ldots,\mathcal{C}_m$ are the connected components of
    $L(\lambda^* - \epsilon)$.
      \end{itemize}
	  \end{definition}

In
addition to being well-separated, we further
require the clusters to be {\it thick}, in a sense made precise below.

  \begin{definition}
A subset $A$ is $h$-thick if $A_{-h} \neq \emptyset$.
  \end{definition}
%The last condition can be relaxed to the condition 
%\[
%    \min \left\{  \min_{k \neq i} \mathrm{dist}(\mathcal{C}_k, \mathcal{C}_i),
%    \min_{k \neq j} \mathrm{dist}(\mathcal{C}_k, \mathcal{C}_j)\right\}> \sigma.
%\]
%where $A \in \mathcal{C}_i$ and $A' \in \mathcal{C_j}$.

 \begin{remark}{\bf Comparison with the separation criterion of \cite{chaudhuri2014consistent}.} 
The notion of $(\epsilon,\sigma)$-separated clusters is analogous to the
corresponding notion of separated clusters introduced in \cite{chaudhuri2014consistent}, with $\epsilon$ and
$\sigma$ quantifying the degree of ``vertical" and ``horizontal" separation
among clusters. There are however, two main differences. First, in our
definition the parameter $\epsilon$ is on the same scale as the density $p$ and,
therefore, represents  vertical separation among level sets in an additive and
not  multiplicative way.
Secondly, we use different parameters to measure the degree of horizontal
separation among clusters ($\sigma$) and the degrees of thickness of the
clusters ($h$); in contrast, \cite{chaudhuri2014consistent} rely on just one parameter ($\sigma$ in their
notation) to express both separation and thickness.
\end{remark}

With these general notions of thick and well-separated clusters in place, we
provide the following uniform consistency results for clustering, which applies
to arbitrary densities.

\begin{corollary}
Let $\{X_1,\ldots,X_n\}$ be an i.i.d. sample from a probability distribution $P$ with an arbitrary density $p$ and let $\epsilon>0$ and $\sigma>0$. Set $a_n=C_1 \sqrt{\frac{\log(n) +\log (1/ h)}{nh^d}}$  as in \eqref{eq:an} and suppose the
	input parameters $h$ satisfying
	\begin{equation}
	\sigma/4\ge h \ge C\left(\frac{\log(n)}{n \epsilon^2 }\right)^{1/d} \quad
	\text{and},
	\end{equation}
for any $C>0$ such that $2 a_n < \epsilon$. 
	Then   with probability at least $1-1/n$, uniformly over all
clusters $A'$ and $A$ that are $h$ thick and $(\epsilon, \sigma)$ separated, 
\begin{itemize}
		\item[i.] $A_{-h} \cap \{\ X_1,\ldots,X_n\}$ and $A'_{-h} \cap\{\ X_1,\ldots,X_n\}$, if non-empty, belong to distinct connected components of
$\mathbb{G}_{k,h}$;
\item[ii.]  all the sample points in  $A_{-h}$, if any, belong to the
		same connected component of $\mathbb{G}_{k,h}$
	\end{itemize}
	where $k = \lceil n h^d
	V_d\lambda \rceil$,	with  $\lambda \in (\lambda^* -\epsilon + a_n,\lambda^*- a_n ]$ and $\lambda^* = \inf_{x A \cup A'}f(x)$.
\end{corollary}

The proof of the previous corollary is almost the identical to the proof of \cite{chaudhuri2014consistent}
and is omitted.

\begin{remark} {\bf Optimality.}
The scaling of the parameters $(\epsilon,\sigma)$ is minimax optimal,
since
the construction used in \cite{chaudhuri2014consistent} yields 
$(\epsilon,\sigma)$-separated clusters that are also $h$ thick with $h =
\sigma/4$,  so that the
resulting lower bound applies to this case as well. 
Thus, DBSCAN delivers nearly minimax optimal cluster
consistency with respect to our definitions of separated and thick clusters. 
\end{remark}

\end{document}